\documentclass[11pt,leqno]{amsart}%
\usepackage{amsmath}
\usepackage{amsfonts}
\usepackage{amssymb}
\usepackage{graphicx}
\usepackage{xcolor}
\usepackage{hyperref}
\usepackage[margin=1.90cm]{geometry}
\providecommand{\U}[1]{\protect\rule{.1in}{.1in}}
\newtheorem{theorem}{Theorem}[section]
\newtheorem*{acknowledgement*}{Acknowledgements}

\newtheorem{corollary}[theorem]{Corollary}

\newtheorem{definition}[theorem]{Definition}

\newtheorem{lemma}[theorem]{Lemma}

\newtheorem{problem}[theorem]{Problem}

\newtheorem{proposition}[theorem]{Proposition}
\newtheorem{remark}[theorem]{Remark}

\newcommand{\B}{\mathbb{B}}
\newcommand{\tg}{{\tilde g}}

\title[Density problems for second order Sobolev spaces]{Density problems for second order Sobolev spaces and cut-off functions on manifolds with unbounded geometry}

\author[Debora Impera]{Debora Impera}
\address[Debora Impera]{Dipartimento di Scienze Matematiche "Giuseppe Luigi Lagrange", Politecnico di Torino, Corso Duca degli Abruzzi, 24, Torino, Italy, I-10129}
\email{debora.impera@gmail.com}

\author[Michele Rimoldi]{Michele Rimoldi}
\address[Michele Rimoldi]{Dipartimento di Scienze Matematiche "Giuseppe Luigi Lagrange", Politecnico di Torino, Corso Duca degli Abruzzi, 24, Torino, Italy, I-10129}
\email{michele.rimoldi@polito.it}

\author {Giona Veronelli}
\address[Giona Veronelli]{Dipartimento di Matematica e Applicazioni, Universit\`a di Milano Bicocca, via R. Cozzi 53, I-20126 Milano, Italy}
\email{giona.veronelli@unimib.it}

\begin{document}
\begin{abstract}
We consider complete non-compact manifolds with either a sub-quadratic growth of the norm of the Riemann curvature, or a sub-quadratic growth of both the norm of the Ricci curvature and the squared inverse of the injectivity radius. We show the existence on such a manifold of a distance-like function with bounded gradient and mild growth of the Hessian. As a main application, we prove that smooth compactly supported functions are dense in $W^{2,p}$. The result is improved for $p=2$ avoiding both the upper bound on the Ricci tensor, and the injectivity radius assumption. As further applications we prove new disturbed Sobolev and Calder\'on-Zygmund inequalities on manifolds with possibly unbounded curvature and highlight consequences about the validity of the full Omori-Yau maximum principle for the Hessian.
\end{abstract}

%\date{\today}

\maketitle
\tableofcontents

\section{Introduction and main results}

Let $(M^m,g)$ be a smooth, complete, possibly non-compact, Riemannian manifold without boundary. For $p\in[1,\infty)$ and $k\geq 2$, denote by $W^{k,p}(M)$ the space of functions on $M$ whose (weak) derivatives of order $0$ to $k$ have a finite $L^p$ norm. Moreover, let $W_{0}^{k,p}(M)$ be the closure of $C_{c}^{\infty}(M)$ in $W^{k,p}(M)$.
\medskip

A classical result in geometric analysis states that for any complete Riemannian manifold, $W_{0}^{1,p}(M)=W^{1,p}(M)$ for any $p\in[1,\infty)$, \cite{aubin-bull}. In this paper, we are interested in the following 

\begin{problem}
 Under which (geometric) assumptions on $M$ does one have that $W_{0}^{2,p}(M)=W^{2,p}(M)$?
\end{problem}

Classical results on this topic can be found in \cite{Aubin}, \cite{HebeyCourant} and references therein. In the following proposition we collect the most up-to-date achievements: point (I) was shown by E. Hebey, \cite[Theorem 2.8]{Hebey}; point (II) was proved by B. G\"uneysu in \cite[Proposition III.18]{Guneysu-book}; point (III) is due to L. Bandara, \cite{Bandara} (for an alternative proof see also \cite[Proposition III.18]{Guneysu-book}).

\begin{proposition}\label{StateArtDensPbms}
Let $(M^{m}, g)$ be a complete Riemannian manifold.
\begin{itemize}
\item[(I)] If $|\mathrm{Ric}_{g}|\leq C$ for some constant $C\geq 0$ and $\mathrm{inj}_{g}(M)>0$, then for every $p\in \left[1,\infty\right)$ we have $W_{0}^{2,p}(M)=W^{2,p}(M)$.
\item[(II)] If $|\mathrm{Riem}_{g}|\leq C$ for some constant $C\geq 0$, then for every $p\in \left[1,\infty\right)$ we have $W_{0}^{2,p}(M)=W^{2,p}(M)$.
\item[(III)] If $\mathrm{Ric}_{g}\geq -C$ for some constant $C\geq 0$ (no assumptions on the injectivity radius!) then $W_{0}^{2,2}(M)=W^{2,2}(M)$.
\end{itemize}
\end{proposition}
\medskip

Often in the applications it is useful to relax the assumptions on the geometry of the manifold, allowing to the bounds on the curvature and on the injectivity radii to be more flexible. The main purpose of this paper is to investigate density problems for second order Sobolev spaces under not necessarily constant bounds on the curvature and (when it is the case) letting the injectivity radii suitably decay at infinity. In particular we obtain the following     

\begin{theorem}\label{th_main1}
Let $(M, g)$ be a complete Riemannian manifold and $o\in M$ a fixed reference point. Set $r(x)\doteq \mathrm{dist}_{g}(x,o)$. Suppose that one of the following set of assumptions holds
\begin{itemize}
\item[(a)] for some $i_0>0$ and $D>0$,
\[
\ |\mathrm{Ric}_{g}|(x)\leq D^2(1+r(x)^2),\quad\mathrm{inj}_{g}(x)\geq \frac{i_0}{D(1+r(x))}>0\quad\mathrm{on}\,\,M.
\]
\item[(b)] for some $D>0$, 
\[
\ |\mathrm{Sect}_{g}|(x)\leq D^2(1+r(x)^2).
\]
\end{itemize}
Then, for every $p\in\left[1, \infty\right)$, we have $W^{2,p}_{0}(M)=W^{2,p}(M)$
\end{theorem}

Moreover, in the special case $p=2$, we can obtain the following improvement of \cite[Theorem 1.1]{Bandara}, where neither an upper bound on the Ricci curvature nor the assumption on the injectivity radii are required. As it is customary in this case the Bochner formula plays a key role.
\begin{theorem}\label{Dens-p=2}
Let $(M,g)$ be a complete Riemannian manifold, $o\in M$, $r(x)\doteq \mathrm{dist}_{g}(o,x)$, and suppose that for some $D>0$
\[
\ \mathrm{Ric}_{g}(x)\geq - D^2 (1+r(x)^2).
\]
Then $W^{2,2}_{0}(M)=W^{2,2}(M)$.
\end{theorem}

%\noindent In another direction, note that density problems were recently investigated in connection to the study of Kato square root type problems in the presence of geometry; see e.g. \cite{Bandara}, \cite{BandaraMcIntosh}.
%\medskip

%To prove our density results we employ the method introduced in \cite{Guneysu}, \cite{GuneysuPigola}. The key step in the proof is the construction on the manifold of special sequences of cut-off functions, with a suitable control on the gradient and on second order derivatives.  In this regard let us notice that, as a matter of fact, both point (II) and point (III) in Proposition \ref{StateArtDensPbms} can be seen as consequences of the existence of such sequences of cut-off functions under the assumptions at hand. More precisely, the result in point (II) uses point (iii) of Proposition \ref{Cutoffs} below, while the result in point (III) can be seen as a consequence of the general criterion (point (a)) given in Proposition \ref{PropGP3.6}, Theorem B in \cite{GuneysuPigola}, and point (i) in Proposition \ref{Cutoffs}.

To prove our density results we employ the method introduced in \cite{Guneysu}, \cite{GuneysuPigola}. The key step in the proof is the construction on the manifold of special sequences of cut-off functions, with a suitable control on the gradient and on second order derivatives.  In this regard let us notice that, as a matter of fact, Proposition \ref{StateArtDensPbms} can be seen as a consequence of the existence of such sequences of cut-off functions under the assumptions at hand. More precisely, the result in point (II) uses point (iii) of Proposition \ref{Cutoffs} below, while the result in point (III) can be seen as a consequence of the general criterion (point (a)) given in Proposition \ref{PropGP3.6}, Theorem B in \cite{GuneysuPigola}, and point (i) in Proposition \ref{Cutoffs}. Finally, also point (I) can be proved using the cut-off functions given by point (iv) in Proposition \ref{Cutoffs} together with Proposition \ref{PropGP3.6} (b)\footnote{Note however that the original proof by Hebey uses a different argument based on a delicate covering technique.}.

Note that such cut-off functions can be tailored starting from suitable smooth exhaustion functions whose gradient and Hessian are controlled in terms of explicit functions of the distance from a fixed reference point. Recall that a smooth function $\rho:M\to\mathbb{R}$ on a Riemannian manifold $(M,g)$ is said to be an exhaustion function if, for every $a\in\mathbb{R}$, the sublevel sets $M_{\rho}(a)=\left\{x\in M:\,\rho(x)<a\right\}$ are relatively compact. In this direction, in this paper we prove the following

\begin{theorem}\label{HCOSubQuadr_coro}
Let $(M, g)$ be a complete Riemannian manifold and $o\in M$ a fixed reference point, $r(x)\doteq \mathrm{dist}_{g}(x,o)$.
Suppose that one of the following set of assumptions holds
\begin{itemize}
\item[(a)] for some $0<\eta\leq1$, some $D>0$ and some $i_0>0$,
\[
\ |\mathrm{Ric}_{g}|(x)\leq D^2(1+r(x)^2)^{\eta},\quad\mathrm{inj}_{g}(x)\geq \frac{i_0}{D(1+r(x))^{\eta}}>0\quad\mathrm{on}\,\,M.
\]
\item[(b)] for some $0<\eta\leq1$ and some $D>0$,
\[
\ |\mathrm{Sect}_{g}|(x)\leq D^2(1+r(x)^2)^\eta.
\]
\end{itemize}
Then there exists an exhaustion function $H\in C^{\infty}(M)$ such that for some positive constant $C>1$ independent of $x$ and $o$, we have on $M$ that 
\begin{itemize}
\item[(i)] $H$ is a distance-like function, i.e., $C^{-1}r(x)\leq H(x)\leq C\max\left\{r(x), 1\right\}$;
\item[(ii)] $|\nabla H|(x)\leq C$;
\item[(iii)] $\left|\mathrm{Hess}\,H\right|(x)\leq C\max\{r(x)^{\eta},1\}$.
\end{itemize}
\end{theorem}

To obtain distance-like exhaustion functions with controlled gradient and Hessian, the previous strategy introduced in \cite{Tam} and adopted also in \cite{RimoldiVeronelli} was the following. One starts with a distance-like function with bounded gradient (which always exists on complete manifolds, \cite{GW}) and let it evolve under the heat flow on $M$. The evolution at a fixed positive time (say $t=1$) preserves the linear growth of the initial datum, as well as the boundedness of the gradient. Moreover Euclidean parabolic Schauder estimates, applied in harmonic coordinates charts of fixed radius centered at any $x\in M$, permit to control the $L^\infty$-norm of the Hessian. However, when the Ricci curvature is unbounded and the injectivity radius is possibly null, the estimates in the heat flow method are difficult to implement, and the parabolic method apparently does not permit to get Theorem \ref{HCOSubQuadr_coro} in its more general assumptions. Accordingly, we use here a different strategy. 

The starting point is a recent result established by D. Bianchi and A. G. Setti, \cite{BianchiSetti}, where exhaustion functions with controlled gradient and Laplacian are constructed on manifolds with Ricci curvature bounded from below by a possibly unbounded non-positive function of the distance from a fixed reference point, without any assumption on the injectivity radius. As in Tam's result, our strategy is then, roughly speaking, to use harmonic coordinates  in order to gain a control on the whole Hessian of these exhaustion functions. An application of elliptic Schauder estimates, Sobolev embeddings and a local Calder\'on-Zygmund inequality permits then to conclude the proof. Note that this latter part is technically more involved than in the parabolic case, since we have to estimate solutions of a semilinear (elliptic) equation instead of a homogenous (parabolic) equation. To deal with the non-uniform bounds on $\mathrm{Ric}$ and $\mathrm{inj}$, everything is done locally, in a suitable ball, with radius decaying at infinity, where we can guarantee the existence of harmonic coordinates with respect to which we have a good control on the metric.

We mention that these techniques can be naturally extended to study density problems for higher order Sobolev spaces on manifolds with unbounded geometry. These results will be presented in the forthcoming paper \cite{HigherOrder}.\medskip

As a further application, the distance-like function $H$ exhibited in Theorem \ref{HCOSubQuadr_coro} permits to deduce the validity of a disturbed Sobolev inequality on non-compact manifolds with possibly unbounded Ricci curvature and possibly vanishing global injectivity radius. It is well known that on a complete non-compact manifold with Ricci curvature  bounded from below and a strictly positive lower bound on $\mathrm{vol}(B_1(x))$ uniform in $x$, one has the continuous embedding $W^{1,p}(M)\subset L^{pm/(m-p)}(M)$, \cite{Varo}. By a result of Croke, the assumption on the volumes of unitary balls is implied by a positive lower bound on the injectivity radius, \cite[Proposition 14]{croke}. Under a conformal change of the Riemannian metric the Ricci curvature modifies following an equation which involves the gradient and the Hessian of the conformal factor. Accordingly, in the assumption of Theorem \ref{HCOSubQuadr_coro} we can use the distance-like function $H$ to get a metric $\tilde g$ in the same conformal class of $(M,g)$ with bounded Ricci curvature and a lower bound on the volumes of unitary balls\footnote{Because of \cite{Mul}, this result is true without curvature and injectivity radius assumptions. The main achievement here is the second order control of the conformal factor.}. Moving back from $\tilde g$ to $g$, we deduce a Sobolev-type inequality on $(M,g)$.

\begin{theorem}\label{th_sob}
Let $(M^m,g)$ be a smooth, complete non-compact Riemannian manifold without boundary. Let $o\in M$, $r(x)\doteq \mathrm{dist}_{g}(x,o)$ and suppose that for some $0<\eta\leq1$, $D>0$ and some $i_{0}>0$,
\begin{equation*}%\label{HpSob}
\ |\mathrm{Ric}_{g}|(x)\leq D^2(1+r(x)^2)^{\eta},\quad\mathrm{inj}_g(x)\geq \frac{i_{0}}{D(1+r(x))^\eta}.
\end{equation*}
Let $p\in[1,m)$ and $q\in [p,mp/(m-p)]$. Then there exist constants $A_1>0$, $A_2>0$, depending on $m$, $p$, $q$ and the constant $C$ from Theorem \ref{HCOSubQuadr_coro}, such that for all $\varphi\in C^\infty_c(M)$ it holds
\begin{align}\label{w-sob}
\left(\int_M |\varphi|^{q}d\mathrm{vol}_{g}\right)^{\frac{1}{q}}
&\leq A_1 \left(\int_M |\nabla \varphi|^p d\mathrm{vol}_{g}\right)^{1/p} + A_2 \left(\int_M \max\{1;r^{2\eta}\} |\varphi|^p d\mathrm{vol}_{g}\right)^{1/p}
\end{align}
\end{theorem}
Disturbed Sobolev inequalities were obtained in the original paper by Varopoulos, \cite{Varo}, and subsequently improved by Hebey, \cite{Hebey}. Namely, they proved that if $\mathrm{Ric_g}\geq - (m-1)D$ for some positive constant $D$, then
\begin{align}\label{w-sob-HV}
\left(\int_M \varphi^{\frac{mp}{m-p}} v^\alpha d\mathrm{vol}_{g}\right)^{\frac{m-p}{mp}}
&\leq A \left( \int_M  |\nabla \varphi|^p v^\beta d\mathrm{vol}_{g}\right)^{\frac 1p} + B \left( \int_M \varphi^p v^\beta d\mathrm{vol}_{g}\right)^{\frac 1p},
\end{align}
where $1\leq p < m$, $\alpha$ and $\beta$ are real constants satisfying $\beta/p - \alpha(m-p)/(mp) \geq 1/m$, and $v(x)\doteq \left(\mathrm{vol}_{g}(B_1(x))\right)^{-1}$.
Combining \eqref{w-sob-HV} with Theorem \ref{HCOSubQuadr_coro} and the conformal method described above, we get a quite general family of Sobolev-type inequalities which contains the one in Theorem \ref{th_sob} when $q=mp/(m-p)$; see Theorem \ref{th_sob_dist}. As a special case of Theorem \ref{th_sob_dist}, one has also the validity of \eqref{w-sob} provided that $\ |\mathrm{Sect}_{g}|(x)\leq D^2(1+r(x)^2)^{\eta}$ and 
\begin{equation}\label{ass_vol}
\mathrm{vol}_g(B_{r^{-\eta}(x)}(x))\geq \frac{E}{(1+r(x))^{m\eta}}.\end{equation}
It is a natural question whether \eqref{ass_vol} together with $\mathrm{Ric}_g\gtrsim -r^{2\eta} $ would suffice to prove Theorem \ref{th_sob}. Note that  both the upper bound on Ricci and the lower bound on the injectivity radius are used to get harmonic radius estimates. On the other hand, the weight $r^{2\eta}$ in \eqref{w-sob} probably can not be avoided since an unweighted Sobolev inequality would imply a lower bound on $\mathrm{vol}_g(B_1(x))$, \cite{Carron_PhD}.
\medskip

Ideas in the proof of Theorem \ref{th_sob_dist} and Theorem \ref{th_sob} can also be applied to other integral inequalities, permitting for instance to obtain the validity of the following $L^{2}$-Calder\'on-Zygmund inequality with weight in our general assumptions. For some results in the same spirit see also \cite[Section 7]{Amar}.

\begin{theorem}\label{th_CZ_dist}
Let $(M^m,g)$ be a smooth, complete non-compact Riemannian manifold without boundary. Let $o\in M$, $r(x)\doteq \mathrm{dist}_{g}(x,o)$ and suppose  that one of the following curvature assumptions holds
\begin{itemize}
\item[(a)] for some $0<\eta\leq1$, some $D>0$ and some $i_0>0$,
\[
\ |\mathrm{Ric}_{g}|(x)\leq D^2(1+r(x)^2)^{\eta},\quad\mathrm{inj}_{g}(x)\geq \frac{i_0}{D(1+r(x))^{\eta}}>0\quad\mathrm{on}\,\,M.
\]
\item[(b)] for some $0<\eta\leq1$ and some $D>0$,
\[
\ |\mathrm{Sect}_{g}|(x)\leq D^2(1+r(x)^2)^\eta.
\]
\end{itemize}
Then there exist constants $A_1>0$, $A_2>0$, depending on $m$, $\eta$, $D$ and the constant $C$ from Theorem \ref{HCOSubQuadr_coro}, such that for all $\varphi\in C^\infty_c(M)$ it holds
\begin{align}\label{CZ-state}
\||\mathrm{Hess}\,\varphi|_g\|_{L^2}^2 \leq  A_1 \|H^{2\eta}\varphi\|_{L^2}^2 +A_2 \|\Delta \varphi\|_{L^2}^2.
\end{align}
\end{theorem}
An extensive study of $L^{p}$-Calder\'on-Zygmund inequalities (in short CZ(p)) and their interplay with the geometry was initiated in \cite{GuneysuPigola}, and we refer to that paper for an introduction to this topic. Even though CZ(2) is known to hold globally, without even require geodesic completeness,  under a global lower bound on the Ricci curvature, it is in general false without this assumption; see  \cite[Theorem B]{GuneysuPigola}. On the other hand, for $p\neq2$,  $p\in(1,\infty)$, CZ(p) holds if $(M, g)$ has bounded Ricci curvature and a positive injectivity radius; see \cite[Theorem C]{GuneysuPigola}. The proof of this result seems to really depend on a harmonic radius bound, and hence on the bound on the injectivity radius. 
Since in our setting $(M,\tilde{g})$  has bounded Ricci curvature and volume non-collapsing but, as far as we know, its injectivity radius is not necessarily bounded from below on the whole of M, it remains an open question if a CZ(p) with weight similar to \eqref{CZ-state} holds in our assumption when $p\neq 2$, .
\medskip

The organization of this paper is as follows. In Section \ref{SeqCutOffs} we recap some known results about the existence of special sequences of cut-off functions and see how these can be used to obtain density results for second order Sobolev spaces. In Section \ref{HarmCoordResc} we see explicitly how, on a suitable ball centered at each point of a manifold with a non-constant Ricci curvature bound and suitably decaying injectivity radii, we can control in harmonic coordinates the metric. In Section \ref{ContrExhaustFct} we construct good distance-like exhaustion function in this generality, first dealing with sub-quadratic Ricci curvature growth and suitably decaying injectivity radii and then with the situation in which we have sub-quadratic sectional curvature growth (and no assumptions on the injectivity radii). Starting from these exhaustion functions, in Section \ref{HessCO} we construct the cut-off functions needed for the proof of our first density result (Theorem \ref{th_main1}). In Section \ref{p2} we focus on the case $p=2$ and give the proof of Theorem \ref{Dens-p=2}. In this case we are assuming only a quadratic negative lower bound on the Ricci curvature. Making use of the weak Laplacian cut-off functions constructed in \cite{BianchiSetti}, under these assumptions, we are able to prove the density result by applying the divergence theorem to a suitable compactly supported vector field together with the Bochner formula.  In Section \ref{sec_sob} we prove the general Theorem \ref{th_sob_dist} about disturbed Sobolev inequalities and, as a consequence of the proof of this latter, we deduce Theorem \ref{th_sob}. Finally, as suggested by one of the anonymous referees, in Section \ref{FurtherAppl} we discuss some further applications of Theorem \ref{HCOSubQuadr_coro} and of the proof of Theorem \ref{th_sob}. In a first part we deduce geometric conditions ensuring the validity of the full Omori-Yau maximum principle at infinity for the Hessian or the validity of martingale completeness. A final subsection is devoted to the proof of Theorem \ref{th_CZ_dist}

\section{Sequences of Cut-off functions and applications to density problems}\label{SeqCutOffs}

Sequences of Laplacian and Hessian cut-off functions where defined in \cite{Guneysu} and \cite{GuneysuPigola}. Here we will need to introduce also the slightly different notions of weak Laplacian and weak Hessian cut-off functions. Namely
\begin{definition}
A complete Riemannian manifold $(M, g)$ is said to admit a sequence $\left\{\chi_{n}\right\}\subset C_{c}^{\infty}(M)$ of Laplacian cut-off functions, if $\left\{\chi_{n}\right\}$ has the following properties:
\begin{itemize}
\item[(C1)] $0\leq\chi_{n}(x)\leq 1$ for all $n\in\mathbb{N}$, $x\in M$;
\item[(C2)] for all compact $K\subset M$, there is a $n_{0}(K)\in\mathbb{N}$ such that for all $n\geq n_{0}(K)$, one has $\left.\chi_{n}\right|_{K}=1$;
\item[(C3)] $\left\|\nabla\chi_{n}\right\|_{\infty}\to 0$ as $n\to\infty$;
\item[(C4)] $\left\|\Delta \chi_{n}\right\|_{\infty}\to0$ ad $n\to\infty$.
\end{itemize}
Furthermore, $(M, g)$ is said to admit a sequence $\left\{\chi_{n}\right\}\subset C_{c}^{\infty}(M)$ of weak Laplacian cut-off functions, if $\left\{\chi_{n}\right\}$ satisfies (C1), (C2), and there exist constants $A_{1}$, $A_{2}$ such that, for all $n\in\mathbb{N}$,
\begin{itemize}
\item[(C3')] $\left\|\nabla\chi_{n}\right\|_{\infty}\leq \frac{A_{1}}{n}$;
\item[(C4')] $\left\|\Delta\chi_{n}\right\|_{\infty}\leq A_{2}$.
\end{itemize}
\end{definition}

\begin{definition}
$(M,g)$ is said to admit a sequence $\left\{\chi_{n}\right\}\subset C_{c}^{\infty}(M)$ of Hessian cut-off functions, if $\left\{\chi_{n}\right\}$ satisfies (C1), (C2), (C3), and 
\begin{itemize}
\item[(C4'')] $\left\|\mathrm{Hess}(\chi_{n})\right\|_{\infty}\to 0$ as $n\to\infty$.
\end{itemize}
Furthermore, $(M, g)$ is said to admit a sequence $\left\{\chi_{n}\right\}\subset C_{c}^{\infty}(M)$ of weak Hessian cut-off functions, if $\left\{\chi_{n}\right\}$ satisfies (C1), (C2), and there exist constants $A_{1}$, $A_{2}$ such that, for all $n\in\mathbb{N}$,
\begin{itemize}
\item[(C3'')] $\left\|\nabla\chi_{n}\right\|_{\infty}\leq A_{1}$;
\item[(C4''')] $\left\|\mathrm{Hess}\chi_{n}\right\|_{\infty}\leq A_{2}$.
\end{itemize}
\end{definition}

The following proposition (partially taken from \cite{Guneysu-book}) should give the state of the art on the existence of such sequences of cut-off functions: point (i) follows by \cite{SY} (see also \cite{Guneysu} for an alternative proof in the case $C=0$), point (ii) was proved in \cite{BianchiSetti}; point (iii) is a consequence of \cite[Lemma 5.3]{CheegerGromov}, point (iv) was proven in \cite[Corollary 5.1]{RimoldiVeronelli} sharpening a construction given in \cite{Tam}; (v) is a consequence of \cite[Theorem 1.3]{Huang}.
\begin{proposition}\label{Cutoffs}
Let $(M^m,g)$ be a complete Riemannian manifold.
\begin{itemize}
\item[(i)] If $(M,g)$ has $\mathrm{Ric}_{g}\geq -C$ for some constant $C\geq 0$, then $M$ admits a sequence of Laplacian cut-off functions.
\item[(ii)] More generally, fix a reference point $o$ in $(M,g)$, and denote by $r(x)\doteq\mathrm{dist}_{g}(x,o)$. If $$\mathrm{Ric}_{g}\geq-(m-1)C^{2}(1+r^2)^{\eta},$$ 
with $\eta\in\left[-1,1\right]$ then, for every $R\geq 1$ when $\eta\in\left(-1,1\right]$ and for every $R>0$ when $\eta=-1$, and for every $\gamma$ bigger than some constant $\Gamma(\eta, C,m)$ depending only on $\eta$, $C$ and $m$, there exists a sequence $\left\{\phi_{R}\right\}\subset C_{c}^{\infty}(M)$ of cut-off functions such that
\begin{enumerate}
\item $\phi_{R}\equiv1$ on $B_{R}(o)$;
\item $\mathrm{supp}(\phi_{R})\subset B_{\gamma R}(o)$;
\item $|\nabla \phi_{R}|\leq \frac{C_{1}}{R}$;
\item $|\Delta\phi_{R}|\leq \frac{C_{2}}{R^{1-\eta}}$.
\end{enumerate}
In particular, for $\eta\in\left[-1,1\right)$ this is a sequence of Laplacian cut-off functions. 
\item[(iii)] If $\left\|\mathrm{Riem}_{g}\right\|_{\infty}< \infty$ then there exists a sequence of Hessian cut-off functions.
\item[(iv)] If $\left\|\mathrm{Ric}_{g}\right\|_{\infty}<\infty$ and $\mathrm{inj}(M)>0$ then there exists a sequence of Hessian cut-off functions 
\item[(v)] There exist $\varepsilon(m)$ and $\Lambda^{\prime}(m)$ such that if $-\Lambda^{\prime}\leq \mathrm{Ric}_{g}\leq \Lambda$ and $\mathrm{vol}_{g}(B_{1}(x))\geq (1-\varepsilon)\omega_{m}$, for all $x\in M$ and for some $\Lambda\geq 0$, then there exists a sequence of Hessian cut-off functions.
\end{itemize}
\end{proposition}

Following the terminology introduced in \cite{GuneysuPigola}, we recall that a $L^{p}$-Calder\'on-Zygmund inequality (CZ(p)) is said to hold on $(M,g)$ for some $1<p<\infty$ if there are constants $C_{1}\geq0$ and $C_{2}>0$, such that for all $u\in C_{c}^{\infty}(M)$ one has
\begin{equation}\label{CZp}\tag{CZ(p)}
\left\|\mathrm{Hess}(u)\right\|_{L^{p}}\leq C_{1}\left\|u\right\|_{L^{p}}+C_{2}\left\|\Delta u\right\|_{L^{p}}.
\end{equation}
Note that, as in \cite{GuneysuPigola}, here we have left out the case $p=1$, since such an inequality indeed fails for the Euclidean Laplace operator in $\mathbb{R}^{m}$.

The following result was proven in \cite{Guneysu}; see also Proposition 3.6 in \cite{GuneysuPigola}.
\begin{proposition}[Theorem 2.6 in \cite{Guneysu} and Proposition 3.6 in \cite{GuneysuPigola}]\label{PropGP3.6}(a) Assume that \eqref{CZp} holds for some $1<p<\infty$ and that $M$ admits a sequence of Laplacian cut-off functions. Then one has that $W_{0}^{2,p}(M)=W^{2,p}(M)$.\\
(b) If $M$ admits a sequence of weak Hessian cut-off functions, then one has $W_{0}^{2,p}(M)=W^{2,p}(M)$ for all $1<p<\infty$.
\end{proposition}

\begin{remark}\label{rmk_dens}
\rm{Actually, what is asked in point (b) of Proposition 3.6 in \cite{GuneysuPigola} is the existence of a sequence of genuine Hessian cut-off functions. Here we observe that what is really needed for the density result is that the gradient and the Hessian of the cut-offs are uniformly bounded. Indeed, first note that $C^{\infty}(M)\cap W^{2,p}(M)$ is dense in $W^{2,p}(M)$ (see for instance \cite[Theorem 2]{GuidettiGueneysuPallara}). Then, given a smooth $f\in W^{2,p}(M)$,  pick a sequence $\left\{\chi_{n}\right\}$ of weak Hessian cut-off functions and define $f_{n}\doteq\chi_{n}f$. Proceeding as in \cite{GuneysuPigola}, we get that 
\begin{align}
\|(f_n-f)\|_{L^p}&=\|((1-\chi_n)f)\|_{L^p}\label{conv1}\\
\|\nabla (f_n-f)\|_{L^p}&\leq \|f\nabla\chi_n\|_{L^p}+\|(1-\chi_n)\nabla f\|_{L^p}\label{conv2}\\
\|\mathrm{Hess}(f_n-f)\|_{L^p}&\leq \|f\mathrm{Hess}(\chi_n)\|_{L^p}+\||\nabla \chi_n| |\nabla f|\|_{L^p}+\|(1-\chi_n)\mathrm{Hess}(f)\|_{L^p}\label{conv3}
\end{align}
Each of $(1-\chi_n)$, $\nabla\chi_n$ and $\mathrm{Hess}(\chi_n)$ is uniformly bounded and supported in $\mathrm{supp}(1-\chi_n)$. Moreover by property (C2), given any compact set $K\subset M$, we have that $\mathrm{supp}(1-\chi_n)\subset M\setminus K$ for $n$ large enough. Since $f\in W^{2,p}(M)$ this permits to conclude that all the terms at the RHS of \eqref{conv1}, \eqref{conv2} and \eqref{conv3} tend to $0$ as $n\to\infty$.
%
%By property (C2) we have that, given a compact set $K\subset M$, there exists a $n_{0}(K)$ such that, for $n>n_{0}(K)$, $\nabla \chi_{n}$ and $\mathrm{Hess}\chi_{n}$ are supported in $M\setminus K$. Considering an exhaustion of $M$ by relatively compact set and using that $f\in W^{2,p}(M)$ and properties (C3'') and (C4'') of $\left\{\chi_{n}\right\}$, a diagonal argument yields a subsequence $\left\{f_{n_j}\right\}$ such that $f_{n_{j}}$ converges to $f$ in $W^{2,p}(M)$ as $n_j\to \infty$.
}
\end{remark}

\section{Harmonic coordinates and rescalings}\label{HarmCoordResc}

Recall that a local coordinate system $\left\{x^{i}\right\}$ is said to be harmonic if for any $i$, $\Delta_{g}x^{i}=0$. The harmonic radius is then defined as follows.
\begin{definition}
Let $(M^m, g)$ be a smooth Riemannian manifold and let $x\in M$. Given $Q>1$, $k\in\mathbb{N}$, and $\alpha\in\left(0,1\right)$, we define the $C^{k,\alpha}$ harmonic radius at $x$ as the largest number $r_{H}=r_{H}(Q,k,\alpha)(x)$ such that on the geodesic ball $B_{r_H}(x)$ of center $x$ and radius $r_{H}$, there is a harmonic coordinate chart such that the metric tensor is $C^{k,\alpha}$ controlled in these coordinates. Namely, if $g_{ij}$, $i,j=1,\ldots,m$, are the components of $g$ in these coordinates, then
\begin{enumerate}
\item $Q^{-1}\delta_{ij}\leq g_{ij}\leq Q\delta_{ij}$ as bilinear forms;
\item $\sum_{1\leq|\beta|\leq k}r_{H}^{|\beta|}\sup\left|\partial_{\beta}g_{ij}(y)\right|+\sum_{|\beta|=k}r_{H}^{k+\alpha}\sup_{y\neq z}\frac{\left|\partial_{\beta}g_{ij}(z)-\partial_{\beta}g_{ij}(y)\right|}{d_{g}(y,z)^{\alpha}}\leq Q-1$.
\end{enumerate}
We then define the (global) harmonic radius $r_{H}(Q,k,\alpha)(M)$ of $(M,g)$ by
\[
\ r_{H}(Q,k,\alpha)(M)=\inf_{x\in M}r_{H}(Q,k,\alpha)(x)
\]
where $r_{H}(Q,k,\alpha)(x)$ is as above.
\end{definition}
As a consequence of \cite[Lemma 2.2]{Anderson} we have the validity of the following 
\begin{proposition}\label{HarmRadEst}
Let $\alpha\in(0,1)$, $Q>1$, $\delta>0$. Let $(M^m,g)$ be a smooth Riemannian manifold, and $\Omega$ an open subset of $M$. Set
\[
\ \Omega(\delta)=\left\{x\in M\quad\mathrm{s.t.}\quad d_{g}(x,\Omega)<\delta\right\}.
\]
Suppose that
\[
\ |\mathrm{Ric}_{g}(x)|\leq 1\quad\mathrm{and}\quad\mathrm{inj}_{g}(x)\geq i\quad\mathrm{for\,\, all}\quad x\in\Omega(\delta),
\]
then, there exists a positive constant $C_{HR}=C_{HR}(m,Q,k,\alpha,\delta, i)$, such that for any $x\in\Omega$
$$r_{H}(Q, 1,\alpha)(x)\geq C_{HR}.$$
\end{proposition}
Since 
\[
\ \partial_{s}g^{ij}=-\partial_{s}g_{lk}g^{il}g^{kj},
\]
note that under the assumptions of Proposition \ref{HarmRadEst}, for every $x\in \Omega$, on $B_{r_{H}}(x)$ we have also that:
\begin{itemize}
\item[(1')] $Q^{-1}\delta^{ij}\leq g^{ij}\leq Q\delta^{ij}$;
\item[(2')] $\sum_{s=1}^mr_{H}\sup\left|\partial_{s}g^{ij}(y)\right|+\sum_{s=1}^m r_{H}^{1+\alpha}\sup_{y\neq z}\frac{\left|\partial_{s}g^{ij}(z)-\partial_{s}g^{ij}(y)\right|}{d_{g}(y,z)^{\alpha}}\leq C(Q)$,
\end{itemize}
for some constant $C(Q)>0$, depending only on $Q$.
\bigskip

Fix now $o\in M$, denote by $r(x)\doteq d_{g}(x, o)$ and assume that,  for some non-decreasing $\lambda:\mathbb{R}\to\mathbb{R}^{+}$ and some uniform constant $i_{0}>0$,
\[
\ |\mathrm{Ric}_{g}(x)|\leq \lambda^{2}(r(x))\quad\mathrm{and}\quad\mathrm{inj}_{g}(x)\geq \frac{i_{0}}{\lambda(r(x))}\quad\mathrm{on}\quad M.
 \]
 We are going to suitably rescale the metric $g$ in order to be able to apply Proposition \ref{HarmRadEst}.
 
Given $x\in M\setminus \bar{B}^{g}_{2}(o)$, for any $y\in B^{g}_{1}(x)$ we have that 
 \[
 \ |\mathrm{Ric}_{g}(y)|\leq \lambda^2(r(x)+1)\quad\mathrm{and}\quad\mathrm{inj}_{g}(y)\geq \frac{i_{0}}{\lambda(r(x)+1)}.
 \]
  Denoting by 
 \[
 \ \lambda_1\doteq\lambda(r(x)+1), 
 \]
 we introduce the rescaled metric
\[
\ g_{\lambda}(y)=\lambda_{1}^{2}g(y).
\]
Then, for any $y\in B_{\lambda_{1}}^{g_{\lambda}}(x)$, 
\[
|\mathrm{Ric}_{g_{\lambda}}(y)|\leq 1\quad\mathrm{and}\quad\mathrm{inj}_{g_{\lambda}}(y)\geq \lambda_{1}\mathrm{inj}_{g}(y)\geq\lambda_{1}\frac{i_{0}}{\lambda_{1}}= i_{0}.
\]
By Proposition \ref{HarmRadEst} we have that there exists a constant $C_{HR}(m,Q,\alpha,\delta, i_0)>0$ such that $\forall\, y\in B^{g_{\lambda}}_{\lambda_{1}-\delta}(x)$, there exist harmonic coordinates  on $B^{g_{\lambda}}_{C_{HR}}(y)$ for which the metric $g_{\lambda}$ satisfies the analogous relations to  (1), (1'), (2) and (2') with $r_{H}=C_{HR}$ (and $k=1$). 
\medskip

Hence, coming back to $g$, for every $y\in B^{g}_{1-\frac{\delta}{\lambda_{1}}}(x)$ we can  find on $B^{g}_{C_{HR}/\lambda_{1}}(y)$ harmonic coordinates with respect to which
\begin{itemize}
\item[(i)] $Q^{-1}\lambda_{1}^{-2}\delta_{ij}\leq g_{ij}\leq Q\lambda_{1}^{-2}\delta_{ij}$;
\item[(ii)] $\sum_{s=1}^m\lambda_{1}^{2}C_{HR}\sup\left|\partial_{s}g_{ij}(y)\right|+\sum_{s=1}^mC_{HR}^{1+\alpha}\lambda_{1}^{2-\alpha}\sup_{y\neq z}\frac{\left|\partial_{s}g_{ij}(z)-\partial_{s}g_{ij}(y)\right|}{d_{g}(y,z)^{\alpha}}\leq Q-1$;
\end{itemize}
and thus also
\begin{itemize}
\item[(i')] $Q^{-1}\lambda_{1}^{2}\delta^{ij}\leq g^{ij}\leq Q\lambda_{1}^2\delta^{ij}$;
\item[(ii')] $\sum_{s=1}^m\lambda_{1}^{-2}C_{HR}\sup\left|\partial_{s}g^{ij}(y)\right|+\sum_{s=1}^mC_{HR}^{1+\alpha}\sup_{y\neq z}\lambda_{1}^{-2-\alpha}\frac{\left|\partial_{s}g^{ij}(z)-\partial_{s}g^{ij}(y)\right|}{d_{g}(y,z)^{\alpha}}\leq C(Q)$,
\end{itemize}
for some constant C(Q).

\section{Construction of controlled exhaustion functions}\label{ContrExhaustFct}
The key step in our construction of sequences of (weak) Hessian cut-off functions is to exhibit suitable smooth exhaustion functions with a good explicit control on the gradient and the Hessian in terms of the distance function $r$ to a fixed reference point. It is already known that this is possible when $\left\|\mathrm{Riem}_{g}\right\|_{\infty}<\infty$, \cite{CheegerGromov}, or when $\left\|\mathrm{Ric}_{g}\right\|_{\infty}<\infty$ and $\mathrm{inj}_{g}(M)>0$, \cite{RimoldiVeronelli}. For a further recent result see also \cite{Huang}. Here, we will deal with the situation in which the curvature is controlled by a sub-quadratic function of the distance from a fixed reference point.

\subsection{Sub-quadratic Ricci growth}\label{SubQuadrRic}
Let $o\in M$, $r(x)\doteq \mathrm{dist}_{g}(x,o)$ and suppose that we are in the assumption (a) of Theorem \ref{HCOSubQuadr_coro}. Up to change the values of the constants $D$ and $i_0$, this is  equivalent to assume that for some $0<\eta\leq1$, $D>0$ and some $i_{0}>0$,
\begin{equation}\label{Hp}
\ |\mathrm{Ric}_{g}|(x)\leq D^2(1+r(x)^2)^{\eta}\doteq\lambda^{2}(r(x)),\quad\mathrm{inj}_{g}(x)\geq\frac{i_{0}}{\lambda(r(x))},\quad\mathrm{on}\, \,M.
\end{equation}

\textit{All over this section, $C$ will denote real constants greater than $1$, all independent of $x\in M$, whose explicit value can possibly change from line to line. Moroever, for any $\beta>0$,  the Euclidean ball of radius $\beta$ centered at the origin will be denoted by $\mathbb{B}_{\beta}$ .}
\medskip

\noindent\textsc{\underline{Step 0}:} \textit{Exhaustion functions with controlled Laplacian.} 
\medskip

Let $h\in C^{\infty}(M)$ be the exhaustion function given in \cite[Theorem 2.1]{BianchiSetti}. Then 
\begin{itemize}
\item[(i)] $C^{-1}r(x)^{1+\eta}\leq h(x)\leq C\max\left\{r(x)^{1+\eta}, 1\right\}$ on $M$;
\item[(ii)] $|\nabla h|\leq Cr^{\eta}$ on $M\setminus\bar{B}_{1}(o)$;
\item[(iii)] $|\Delta h|\leq Cr^{2\eta}$ on $M\setminus\bar{B}_{1}(o)$.
\end{itemize}
Moreover, by the construction in the proof of \cite[Theorem 2.1]{BianchiSetti}, $h$ is a solution of 
\begin{equation}\label{Poisson}
\Delta h=|\nabla h|^2-\theta \tilde r^{2\eta}\doteq f
\end{equation}
on $M\setminus\bar{B}_{1}(o)$, where $\theta$ is a positive fixed constant and $\tilde r\in C^\infty(M)$ is a smooth 1st order approximation of the distance function, which satisfies in particular 
\begin{itemize}
\item $C^{-1}r(x)\leq \tilde r(x)\leq C\max\left\{r(x), 1\right\}$
\item $|\nabla \tilde r|\leq C$
\end{itemize}
on $M$.
\medskip

%$r$ NON E' LA DISTANZA MA LA DISTANZA LISCIATA. DEFINIRLA COME IN BIANCHI-SETTI E CONTINUARE A DENOTARLA CON $r$. OSSERVARE ANCHE ESPLICITAMENTE  CHE FORMALMENTE NEL SEGUITO CONTINUEREMO AD USARE LA DISUGUAGLIANZA TRIANGOLARE E IL FATTO CHE $|\nabla r|=1$ SENZA PERDERE IN GENERALITA'.

\noindent\textsc{\underline{Step 1}}: \textit{using harmonic coordinates.}
\medskip

%IN TEORIA POTREMMO SCEGLIERE $x$ FUORI DA UN COMPATTO E METTERE ANCHE QUI $1$ AL POSTO DI $\varepsilon$. TU COSA DICI?

Given $x\in M\setminus \bar{B}_{2}(o)$, we define $h_{x}:B_{\varepsilon}(x)\to\mathbb{R}$ by
\begin{equation*}
\ h_{x}(y)=h(y)-h(x).
\end{equation*}
Then $h_{x}(x)=0$, $h_{x}$ satisfies \eqref{Poisson}, and
\begin{itemize}
\item $|\nabla h_{x}|\leq Cr^{\eta}$ ;
\item $|\Delta h_{x}|\leq Cr^{2\eta}$;
\item  $\left|\mathrm{Hess}\,h_{x}\right|=\left|\mathrm{Hess}\,h\right|$ ,
\end{itemize}
Fix now $\alpha\in (0,1)$, an accuracy $Q>1$ and a sufficiently small $\delta>0$. By \eqref{Hp} and Section \ref{HarmCoordResc}, letting again
\[
\ \lambda_{1}\doteq\lambda(r(x)+1)=D((r(x)+1)^2+1)^{\eta/2},
\]
 we know that there exists a constant $C_{HR}(m,Q,\alpha,\delta, i_{0})$ such that we can find on $B_{C_{HR}/\lambda_{1}}(x)$ a harmonic chart
\[
\ \varphi_{H}=(y^{1},\ldots,y^{m}):B_{C_{HR}/\lambda_{1}}(x)\to U\subset\mathbb{R}^{m},
\]
such that $\varphi_{H}(x)=0$, and with respect to which
\begin{itemize}
\item[(i)] $Q^{-1}\lambda_{1}^{-2}\delta_{ij}\leq g_{ij}\leq Q\lambda_{1}^{-2}\delta_{ij}$;
\item[(ii)] $\sum_{s=1}^m\lambda_{1}^{2}C_{HR}\sup\left|\partial_{s}g_{ij}(y)\right|+\sum_{s=1}^m C_{HR}^{1+\alpha}\lambda_{1}^{2-\alpha}\sup_{y\neq z}\frac{\left|\partial_{s}g_{ij}(z)-\partial_{s}g_{ij}(y)\right|}{d_{g}(y,z)^{\alpha}}\leq Q-1$;

\item[(i')] $Q^{-1}\lambda_{1}^{2}\delta^{ij}\leq g^{ij}\leq Q\lambda_{1}^2\delta^{ij}$;
\item[(ii')] $\sum_{s=1}^m\lambda_{1}^{-2}C_{HR}\sup\left|\partial_{s}g^{ij}(y)\right|+\sum_{s=1}^mC_{HR}^{1+\alpha}\lambda_{1}^{-2-\alpha}\sup_{y\neq z}\frac{\left|\partial_{s}g^{ij}(z)-\partial_{s}g^{ij}(y)\right|}{d_{g}(y,z)^{\alpha}}\leq C(Q)$,
\end{itemize}
for some constant $C(Q)$.  

%\noindent Setting $\left\{\hat{y}^{i}\right\}=\left\{\lambda_{\rho} y^{i}\right\}$ then denoting by $\partial_{\hat{i}}=\frac{\partial }{\partial \hat{y}^i}$, since $\partial_{\hat{i}}=\frac{1}{\lambda_{\rho}}\partial_{i}$, we have $g_{\hat{i}\hat{j}}=g(\partial_{\hat{i}},\partial_{\hat{j}})=\frac{1}{\lambda_{\rho}^2}g_{ij}$.
%
%\begin{itemize}
%\item[(i)] $Q^{-1}\delta_{ij}\leq g_{\hat{i}\hat{j}}\leq Q\delta_{ij}$;
%\item[(ii)] $\sum_{\hat{\beta}}\lambda_{\rho}^{2}C_{HR}\sup\left|\partial_{\hat{\beta}}g_{\hat{i}\hat{j}}(y)\right|+\sum_{\hat{\beta}}C_{HR}^{1+\alpha}\lambda_{\rho}^{2-\alpha}\sup_{y\neq z}\frac{\left|\partial_{\hat{\beta}}g_{\hat{i}\hat{j}}(z)-\partial_{\hat{\beta}}g_{\hat{i}\hat{j}}(y)\right|}{d_{g}(y,z)^{\alpha}}\leq Q-1$;
%
%\item[(i')] $Q^{-1}\lambda_{\rho}^{2}\delta^{ij}\leq g^{ij}\leq Q\lambda_{\rho}^2\delta^{ij}$;
%\item[(ii')] $\sum_{\beta}\lambda_{\rho}^{-2}C_{HR}\sup\left|\partial_{\beta}g^{ij}(y)\right|+\sum_{\beta}C_{HR}^{1+\alpha}\sup_{y\neq z}\lambda_{\rho}^{-2-\alpha}\frac{\left|\partial_{\beta}g^{ij}(z)-\partial_{\beta}g^{ij}(y)\right|}{d_{g}(y,z)^{\alpha}}\leq C(Q)$,
%\end{itemize}
\medskip

\noindent\textsc{\underline{Step 2}}: \textit{pointwise Schauder estimate.} 
\medskip

Note that
\[
\ \mathbb{B}_{C_{HR}/\sqrt{Q}}\subset U=\varphi_{H}\left(B_{C_{HR}/\lambda_{1}}(x)\right).
\]
Define $\hat{h}_{x}\doteq h_{x}\circ\varphi_{H}^{-1}$, $\hat{f}\doteq f\circ\varphi_{H}^{-1}$, and $\hat{g}^{ij}\doteq g^{ij}\circ\varphi_{H}^{-1}$. Letting $\beta\doteq C_{HR}/(2\sqrt{Q})$, define $\hat{h}_{\beta,x}:\mathbb{B}_{2}\to\mathbb{R}$ by $\hat{h}_{\beta,x}(v)\doteq\hat{h}_{x}(\beta v)$. Then
\[
\ \partial^{2}_{ij} \hat{h}_{\beta, x}(v)=\beta^{2}\partial^{2}_{ij}\hat{h}_{x}(\beta v).
\]
By \eqref{Poisson}, we hence get that in these coordinates, on $\mathbb{B}_{2}$,
\begin{align*}
\hat{g}^{ij}(\beta v)\partial^{2}_{ij}\hat{h}_{\beta, x}(v)=&\beta^{2}\hat{g}^{ij}(\beta v)\partial^{2}_{ij} \hat{h}_{x}(\beta v)\\
=&\beta^2\hat{f}(\beta v).
\end{align*}
Hence
\begin{equation}\label{EqHathbeta}
\lambda_{1}^{-2}\hat{g}^{ij}(\beta v)\partial^{2}_{ij}\hat{h}_{\beta, x}(v)=\lambda_{1}^{-2}\beta^2\hat{f}(\beta v)\doteq\hat{f}_{\beta}(v).
\end{equation}
Note that
\begin{align*}
&\lambda_{1}^{-2}\hat{g}^{ij}(\beta\cdot)\geq \lambda_{1}^{-2}Q^{-1}\lambda_{1}^2\delta^{ij}=Q^{-1}\delta^{ij}\quad\mathrm{on\,\,}\mathbb{B}_{2},\\
&\left\|\lambda_{1}^{-2}\hat{g}^{ij}(\beta\cdot)\right\|_{L^{\infty}(\mathbb{B}_{2})}\leq Q\lambda_{1}^{-2}\lambda_{1}^2=Q,\\
&\left[\lambda_{1}^{-2}\hat{g}^{ij}(\beta \cdot)\right]_{C^{0,\alpha}_{L^{\infty}}}(0)\leq\frac{\lambda_{1}^{-2}\beta}{C_{HR}}C(Q)\lambda_{1}^2=C(Q)/2\sqrt{Q}.
\end{align*}

%\begin{proposition}\label{PointSchaud}
%Let $L=a^{ij}\partial^{2}_{ij}u$ be a second order elliptic operator in $\mathbb{B}_{1}(0)$. Let us assume that
%\begin{align*}
% a^{ij}(v)\xi_{i}\xi_{j}&\geq \lambda,\quad \ \forall\, \xi\in\mathbb{S}^{m-1}\subset\mathbb{R}^{m},\quad\forall\, v\in \mathbb{B}_{1}(0),\\
% |a^{ij}(v)|&\leq\kappa\quad\forall\, v\in \mathbb{B}_{1}(0),\\
% |a^{ij}(v)-a^{ij}(0)|&\leq K|v|^{\alpha}\quad\forall\,v\in \mathbb{B}_{1}(0),
%\end{align*}
%for some positive constants $\lambda, \kappa, K$ and $\alpha\in (0,1)$. Let $u\in C^{0}(\mathbb{B}_{1})$ be a viscosity solution of $Lu=f$ in $\mathbb{B}_{1}(0)$ for some $f\in L^{\infty}(\mathbb{B}_{1}(0))$. If for some constant $\alpha\in (0,1)$, $f\in C^{0,\alpha}_{L^{\infty}}(0)$, and $a_{ij}\in C^{0,\alpha}_{L^{\infty}}(0)$ for any $1\leq i,j\leq m$, then $u\in C^{2,\alpha}_{L^{\infty}}(0)$. Moreover, the following estimate holds
%\[
%\ \left[u\right]_{C^{2,\alpha}_{L^{\infty}}(0)}\leq C_{HL}\left\{\left\|u\right\|_{L^{\infty}(\mathbb{B}_{1}(0))}+\left\|f\right\|_{L^{\infty}(\mathbb{B}_{1}(0))}+\left[f\right]_{C^{0,\alpha}_{L^{\infty}}}(0)\right\},
%\]
%where $C_{HL}=C_{HL}\left(m,\lambda,\kappa, K,\alpha\right)$.
%\end{proposition}

Applying classical pointwise Schauder estimates for second order elliptic operators (see in particular \cite[Theorem 1.1]{Han} or \cite[Theorem 5.20]{HanLin}) to equation \eqref{EqHathbeta} we hence get that
\begin{equation}\label{SchaudEstHathbeta}
\left[\hat{h}_{\beta,x}\right]_{C^{2,\alpha}_{L^{\infty}}}(0)\leq C \left\{\left\|\hat{h}_{\beta, x}\right\|_{L^{\infty}(\mathbb{B}_{1})}+\left\|\hat{f}_{\beta}\right\|_{L^{\infty}(\mathbb{B}_{1})}+\left[\hat{f}_{\beta}\right]_{C^{0,\alpha}_{L^{\infty}}}(0)\right\}.
\end{equation}
From this it follows that 
\begin{equation*}
\beta^2\left|\partial^2_{ij}\hat{h}_{x}(0)\right|\leq C\left\{\left\|\hat{h}_{ x}\right\|_{L^{\infty}(\mathbb{B}_{\beta})}+\frac{\beta^2}{\lambda_{1}^{2}}\left\|\hat{f}\right\|_{L^{\infty}(\mathbb{B}_{\beta})}+\frac{\beta^{2+\alpha}}{\lambda_{1}^{2}}\left[\hat{f}\right]_{C^{0,\alpha}_{L^{\infty}}(\B_\beta)}\right\},
\end{equation*}
i.e. 
\begin{equation}\label{Est1}
\left|\partial^2_{ij}\hat{h}_{x}(0)\right|\leq C\left\{\frac{1}{\beta^2}\left\|\hat{h}_{ x}\right\|_{L^{\infty}(\mathbb{B}_{\beta})}+\frac{1}{\lambda_{1}^{2}}\left\|\hat{f}\right\|_{L^{\infty}(\mathbb{B}_{\beta})}+\frac{\beta^\alpha}{\lambda_{1}^{2}}\left[\hat{f}\right]_{C^{0,\alpha}_{L^{\infty}}(\B_\beta)}\right\}.
\end{equation}
\medskip

\textit{From now on we will denote $\hat{h}_{x}$ simply by $\hat{h}$, being understood the fact that it depends on the point $x$ we have fixed on $M$.}
\medskip

About the first term on the RHS of \eqref{Est1}, we note that for any $v\in\mathbb{B}_{\beta}$, letting $y=\varphi_{H}^{-1}(v)$, we have that
\begin{align}\label{est_hath}
|\hat{h}(v)|=&|h_{x}(y)|=|h(y)-h(x)|\\
\leq&C d_{g}(x,y)\sup \{r^{\eta}(\zeta):\zeta\in B_{d_g(x,y)}(x)\}\nonumber\\
\leq&C\left(r(x)+\frac{C_{HR}}{\lambda_{1}}\right)^{\eta}\frac{C_{HR}}{\lambda_{1}}\nonumber\\
\leq& C\nonumber.
\end{align}
Here the last inequality comes from the definition of $\lambda_{1}$.

\noindent About the second term, note that 
\begin{align*}
|\hat{f}(v)|=&|f(y)|= \left||\nabla h_x|^2(y)-\theta \tilde r(y)^{2\eta}\right|\\
%\leq&D_{3}r(y)^{2\eta}+Cr(y)^{2\eta}\\
\leq & Cr(y)^{2\eta}\\
\leq&C\left(r(x)+\frac{C_{HR}}{\lambda_{1}}\right)^{2\eta}\\
\leq&  C\rho^{2\eta} r(x)^{2\eta},
\end{align*}
for some positive constant $\rho>1$.
In particular,
\begin{align}\label{9half}
\ \frac{1}{\lambda_{1}^{2}}\left\|\hat{f}\right\|_{L^{\infty}(\mathbb{B}_{\beta})}&\leq C\frac{r(x)^{2\eta}}{\lambda_{1}^{2}}\\&\leq C.\nonumber
\end{align}

It remains to estimate $[\hat{f}]_{C^{0,\alpha}_{L^{\infty}}(\mathbb{B}_{\beta})}$. Letting $v,w\in\mathbb{B}_{\beta}$, note that
\begin{align}
\frac{\left|\hat{f}(v)-\hat{f}(w)\right|}{|v-w|^{\alpha}}=&\frac{||\nabla h_x|^2(\varphi_{H}^{-1}(v))-|\nabla h_x|^2(\varphi_{H}^{-1}(w))-\theta(\tilde r^{2\eta}(\varphi_{H}^{-1}(v))-\tilde r^{2\eta}(\varphi_{H}^{-1}(w)))|}{|v-w|^{\alpha}}\label{EstHoldf}\\
\leq&[|\nabla h_x|^2\circ\varphi_{H}^{-1}]_{C^{0,\alpha}_{L^{\infty}}(\mathbb{B}_{\beta})}+C\frac{|\tilde r^{2\eta}(\varphi_{H}^{-1}(v))-\tilde r^{2\eta}(\varphi_{H}^{-1}(w))|}{|v-w|^{\alpha}}.\nonumber
\end{align}

The first term will be estimated in Step 3 and Step 4 below. About the second term, letting $y=\varphi_{H}^{-1}(v)$ and $z=\varphi_{H}^{-1}(w)$, we have that
\begin{align}
\frac{|\tilde r^{2\eta}(\varphi_{H}^{-1}(v))-\tilde r^{2\eta}(\varphi_{H}^{-1}(w))|}{|v-w|^{\alpha}}=&\frac{|\tilde r(y)^{2\eta}-\tilde r(z)^{2\eta}|}{|v-w|^{\alpha}}\label{SecondAddendHold}\\
\leq&\sup \{2\eta \tilde r(\zeta)^{2\eta-1}|\nabla \tilde r|(\zeta):\zeta\in B_{\max\left\{d_g(x,y),d_{g}(x,z)\right\}}(x)\}\frac{d_g(y,z)}{|v-w|^{\alpha}}\nonumber\\
\leq&2\eta C \sup\{ \tilde r(\zeta)^{2\eta-1}:\zeta\in B_{d_g(x,y)}(x)\}\frac{\sqrt{Q}}{\lambda_{1}}|v-w|^{1-\alpha}\nonumber\\
\leq&2\eta C
\frac{\sqrt{Q}}{\lambda_{1}}(2\beta)^{1-\alpha}\rho^{2\eta-1} r(x)^{2\eta-1},\nonumber
\end{align}
where $Q$ is the chosen accuracy of the harmonic coordinates.
\medskip

\noindent\textsc{\underline{Step 3}}: \textit{estimate of $[|\nabla h_x|^2\circ \varphi_{H}^{-1}]_{C^{0,\alpha}_{L^{\infty}}(\mathbb{B}_{\beta})}$ using Sobolev embeddings.}
\medskip

Letting $p>m$, by the Euclidean Sobolev embeddings (see e.g \cite[p. 109]{Adams}), we have 
%the following
%\begin{lemma}
%Let $v:\B_\epsilon\subset \mathbb R^m$ be in $ W^{1,p}(\B_\epsilon)$. Then for any $m<p<\infty$ and $0<\alpha<1-\frac{n}{p}$ there exists $K_3(m,p,\alpha)>0$ such that
%\[
%[v]_{C^\alpha_{L^\infty}(\B_\epsilon)}\leq K_3\epsilon ^{-\alpha-\frac{n}{p}}\left\{\|v\|^p_{L^p(\B_\epsilon)}+\epsilon^p\sum_j\|\partial_j v\|^p_{L^p(\B_\epsilon)}\right\}^{1/p}
%\]
%\end{lemma}
%In particular
\[
[|\nabla h_x|^2\circ\varphi_{H}^{-1}]_{C^{0,\alpha}_{L^\infty}(\B_\beta)}\leq K_3\left\{\||\nabla h_x|^2\circ\varphi_{H}^{-1}\|^p_{L^p(\B_\beta)}+\sum_j\|\partial_j \left( |\nabla h_x|^2\circ\varphi_{H}^{-1}\right)\|^p_{L^p(\B_\beta)}\right\}^{1/p},
\]
with $\alpha=1-\frac{m}{p}$ and $K_3$ a positive constant depending only on $p$, $m$ and $\alpha$.
Concerning the first term in the RHS, since $|\nabla h_x|\leq Cr^{\eta}$ we get 
\begin{align}\label{est_grad}
\||\nabla h_x|^2\circ\varphi_{H}^{-1}\|^p_{L^p(\B_\beta)}&\leq C\|r^{2\eta}\circ\varphi_{H}^{-1}\|^p_{L^p(\B_\beta)}\leq \omega_mC\beta^m\|r\|_{L^\infty(B_\frac{C_{HR}}{2\lambda_{1}}(x))}^{2\eta p}\\&\leq C\left(r(x)+\frac{C_{HR}}{2\lambda_{1}}\right)^{2\eta p},\nonumber
\end{align}
where $\omega_m$ is the volume of the $m$-dimensional unit sphere. Concerning the second term, let us compute
\begin{align*}
\partial_j(|\nabla h_x|^2\circ\varphi_{H}^{-1}) &= \partial_j (\partial_k  h_x\partial_i  h_x g^{ki})\\
&= 2 \partial_j\partial_k  h_x\partial_i  h_x g^{ki} + \partial_k  h_x\partial_i  h_x \partial_j g^{ki} ,
\end{align*} 
so that
\begin{align*}
\|\partial_j (|\nabla h_x|^2\circ\varphi_{H}^{-1})\|_{L^p(\B_\beta)} 
\leq& 2 \|\partial_j\partial_k  \hat{h}\partial_l  \hat{h} \hat{g}^{kl}\|_{L^p(\B_\beta)} + \|\partial_k  \hat{h}\partial_l  \hat{h} \partial_j \hat{g}^{kl}\|_{L^p(\B_\beta)}\\
\leq& 2\sum_{k,l} \|\partial_j\partial_k  \hat{h}\|_{L^p(\B_\beta)}\|\partial_l  \hat{h}\|_{L^\infty(\B_\beta)} Q\lambda_{1}^{2} +\sum_{k,l} \|\partial_k  \hat{h}\|_{L^p(\B_\beta)}\|\partial_l  \hat{h}\|_{L^\infty(\B_\beta)} \frac{C(Q)}{C_{HR}}\lambda_{1}^{2}\\
\leq& 2m^2 \|D^2\hat{h}\|_{L^p(\B_\beta)}\|D  \hat{h}\|_{L^\infty(\B_\beta)} Q\lambda_{1}^{2} +m^2 \|D \hat{h}\|_{L^p(\B_\beta)}\|D \hat{h}\|_{L^\infty(\B_\beta)} \frac{C(Q)}{C_{HR}}\lambda_{1}^{2}\\
\leq& 2m^2 \|D^2\hat{h}\|_{L^p(\B_\beta)}\||\nabla h_x|\circ\varphi_{H}^{-1}\|_{L^\infty(\B_\beta)} Q^{3/2}\lambda_{1} \\
&+m^2 \||\nabla h_x|\circ\varphi_{H}^{-1}\|_{L^p(\B_\beta)}\||\nabla h_x|\circ\varphi_{H}^{-1}\|_{L^\infty(\B_\beta)} \frac{QC(Q)}{C_{HR}},
\end{align*}
where we are denoting by $D\hat{h}$ and $D^{2}\hat h$, respectively, the Euclidean gradient and the Euclidean Hessian of $\hat{h}$, and we have used the fact that
\begin{equation}\label{est_D}
|D\hat h|^2 = \partial_k\hat h\partial_l\hat h \delta^{kl}\leq \lambda_{1}^{-2}Q\partial_k\hat h\partial_l\hat h \hat g^{kl} = \lambda_{1}^{-2}Q|\nabla h_x|^2\circ\varphi_{H}^{-1}.
\end{equation}
Reasoning as in \eqref{est_grad} we get 
\begin{align*}
\|\partial_j (|\nabla h_x|^2\circ\varphi_{H}^{-1})\|_{L^p(\B_\beta)} 
\leq& 2m^2C \|D^2\hat{h}\|_{L^p(\B_\beta)}\|r^\eta\circ\varphi_{H}^{-1}\|_{L^\infty(\B_\beta)} Q^{3/2}\lambda_{1} \\
&+m^2 C\|r^\eta\circ\varphi_{H}^{-1}\|_{L^p(\B_\beta)}\|r^\eta\circ\varphi_{H}^{-1}\|_{L^\infty(\B_\beta)}  \frac{QC(Q)}{C_{HR}}\\
\leq& C\lambda_{1} \left(r(x)+\frac{C_{HR}}{2\lambda_{1}}\right)^\eta \|D^2\hat{h}\|_{L^p(\B_\beta)} +C\left(r(x)+\frac{C_{HR}} {2\lambda_{1}}\right)^{2\eta}.
\end{align*}
\medskip

\noindent\textsc{\underline{Step 4}}: \textit{estimate of $||D^{2}\hat{h}||_{L^{p}(\mathbb{B}_{\beta})}$ by a  Calder\'on-Zygmund inequality.}
\medskip

Let $\phi\in C^\infty_c(\B_{2\beta})$ be such that $0\leq\phi\leq 1$ and $\phi\equiv 1$ on $\B_\beta$ and $\max\{\|\nabla \phi\|_\infty;\|\Delta \phi\|_\infty\}<C_1$ for some $C_1=C_1(\beta,m)\in\mathbb R$. According to the Calder\'on-Zygmund inequality, \cite[Corollary 9.10]{GT}, there exists a constant $C_2=C_2(m,p)>0$ such that
\begin{align}\label{CZ}
\|D^2\hat h\|_{L^p(\B_\beta)}
&\leq \|D^2(\phi\hat h)\|_{L^p(\B_{2\beta})}\\
&\leq C_2 \|\Delta_0(\phi\hat h)\|_{L^p(\B_{2\beta})}\nonumber\\
&\leq C_2\left(\|\phi\Delta_0\hat h\|_{L^p(\B_{2\beta})}+\|\hat h\Delta_0\phi\|_{L^p(\B_{2\beta})}+ 2 \|D\hat h\cdot D\phi\|_{L^p(\B_{2\beta})}\right)\nonumber\\
&\leq C_2\left(\|\Delta_0\hat h\|_{L^p(\B_{2\beta})}+C_1\|\hat h\|_{L^p(\B_{2\beta})}+ 2C_1 \|D\hat h\|_{L^p(\B_{2\beta})}\right),\nonumber
\end{align}
where $\Delta_0=\sum_i \partial_i\partial_i$ is the Euclidean Laplacian.

Now, 
\begin{align*}
|\Delta_0 \hat h| &= |\partial_k\partial_j \hat h  \delta^{kj}|\\
&\leq |\partial_k\partial_j \hat h  \hat g^{kj}|Q\lambda_{1}^{-2}
=Q\lambda_{1}^{-2}|\Delta h_x|\circ\varphi_{H}^{-1}\\
&\leq CQ\lambda_{1}^{-2} \left(r^{2\eta}\circ\varphi_{H}^{-1}\right).
\end{align*}
Accordingly, 
\begin{equation*}
\|\Delta_0\hat h\|_{L^p(\B_{2\beta})}\leq C\lambda_{1}^{-2} \left(r(x)+\frac{C_{HR}}{\lambda_{1}}\right)^{2\eta}.
\end{equation*}
As in \eqref{est_hath} we have that $\|\hat h\|_{L^\infty(\B_{2\beta})}\leq C$, so that $\|\hat h\|_{L^p(\B_{2\beta})}\leq C$. Moreover, using \eqref{est_D},
\[
\|D\hat h\|_{L^p(\B_{2\beta})}\leq
\lambda_{1}^{-1}\sqrt Q\||\nabla h_x|\circ\varphi_{H}^{-1}\|_{L^p(\B_{2\beta})}\leq\lambda_{1}^{-1}C \left(r(x)+\frac{C_{HR}}{\lambda_{1}}\right)^\eta.
\]

Inserting these estimates in \eqref{CZ} gives
\begin{align}
\|D^2\hat h\|_{L^p(\B_\beta)}
\leq& C_2\left[C\lambda_{1}^{-2} \left(r(x)+\frac{C_{HR}}{\lambda_{1}}\right)^{2\eta}+ C\left(1 + \lambda_{1}^{-1} \left(r(x)+\frac{C_{HR}}{\lambda_{1}}\right)^\eta\right)\right]\nonumber\\
\leq& C.\nonumber
\end{align}
Coming back to Step 3, 
\begin{align*}
\|\partial_j (|\nabla h_x|^2\circ\varphi_{H}^{-1})\|_{L^p(\B_\beta)} 
&\leq C r^{2\eta}
\end{align*}
and 
\begin{equation}\label{EstHoldNablah}
[|\nabla h_x|^2\circ\varphi_{H}^{-1}]_{C^{0,\alpha}_{L^{\infty}}(\B_\beta))}\leq C r^{2\eta}.
\end{equation}
\medskip

\noindent\textsc{\underline{Step 5}}: \textit{estimate of $|\mathrm{Hess}\, h|$.}
\medskip

Using \eqref{EstHoldNablah} and \eqref{SecondAddendHold}, we get by \eqref{EstHoldf} that
\[
\ \frac{\beta^{\alpha}}{\lambda_{1}^{2}}[\hat{f}]_{C^{0,\alpha}_{L^{\infty}}(\mathbb{B}_{\beta})}\leq C,
\]
and hence, by \eqref{Est1}, \eqref{est_hath} and \eqref{9half}, 
\[
\ |\partial^{2}_{ij}h_{x}(x)|= |\partial^{2}_{ij}\hat{h}(0)|\leq C.
\]
Recalling now that
\[
\ \nabla_{i}\nabla_{j}h_{x}=\partial^{2}_{ij}h_{x}-\Gamma_{ij}^{k}\partial_{k}h_{x},
\]
we can compute that
\begin{align*}
\left|\mathrm{Hess}\,h_{x}\right|(x)=&\left[g^{ik}g^{jl}\left(\partial^{2}_{ij}h_{x}-\Gamma_{ij}^{s}\partial_{s}h_{x}\right)\left(\partial^{2}_{kl}h_{x}-\Gamma_{kl}^{t}\partial_{t}h_{x}\right)\right]^{\frac{1}{2}}(x)\\
\leq&\left[g^{ik}g^{jl}\partial^{2}_{ij}h_{x}\partial^{2}_{kl}h_{x}+g^{ik}g^{jl}\Gamma_{ij}^{s}\Gamma_{kl}^{t}\partial_{s}h_{x}\partial_{t}h_{x}\right.\\
&\left.-2g^{ik}g^{jl}\partial^{2}_{ij}h_{x}\Gamma_{kl}^{t}\partial_{t}h_{x}\right]^{\frac{1}{2}}(x).
\end{align*}
Since 
\begin{align*}
|g^{ik}g^{jl}\partial^{2}_{ij}h_{x}\partial^{2}_{kl}h_{x}|(x)\leq&Q^{2}\lambda_{1}^{4}\left|\delta^{ik}\delta^{jl}\partial^{2}_{ij}h_{x}(x)\partial^{2}_{kl}h_{x}\right|(x)\\
=&Q^{2}\lambda_{1}^{4}|\partial_{ij}h_{x}|^{2}(x)\leq C\lambda_{1}^{4}\leq Cr^{4\eta}(x),\\
|g^{ik}g^{jl}\Gamma_{ij}^{s}\Gamma_{kl}^{t}\partial_{s}h_{x}\partial_{t}h_{x}|(x)\leq&\frac{9}{4}CQ^{5}\frac{(Q-1)^2}{C_{HR}^{2}}\lambda_{1}^{2}r^{2\eta}(x)\\&\leq Cr^{4\eta}(x),\\
|2g^{ik}g^{jl}\partial^{2}_{ij}h_{x}\Gamma_{kl}^{t}\partial_{t}h_{x}|(x)\leq&2C^2Q^{6}\sqrt{Q}\frac{(Q-1)}{C_{HR}}\lambda_{1}^{4}\lambda_{1}^{-1}r^{\eta}(x)\\&\leq Cr^{4\eta}(x),
\end{align*}
we eventually obtain that
\begin{equation}\label{EstHess}
\left|\mathrm{Hess}\,h\right|(x)=\left|\mathrm{Hess}\, h_x\right|(x)\leq Cr^{2\eta}(x).
\end{equation}
\medskip

We have thus  proved the following
\begin{theorem}\label{HCOSubQuadr}
Let $(M, g)$ be a complete Riemannian manifold and $o\in M$ a fixed reference point, $r(x)\doteq \mathrm{dist}_{g}(x,o)$. Suppose that for some $0<\eta\leq1$, some $D>0$ and some $i_0>0$,
\[
\ |\mathrm{Ric}_{g}|(x)\leq D^2(1+r(x)^2)^{\eta},\quad\mathrm{inj}_{g}(x)\geq \frac{i_0}{D(1+r(x))^{\eta}}>0\quad\mathrm{on}\,\,M.
\]
Then there exists an exhaustion function $h\in C^{\infty}(M)$ such that, for some positive constant $C$ independent of $x$ and $o$, we have that
\begin{itemize}
\item[(i)]$C^{-1}r(x)^{1+\eta}\leq h(x)\leq C\max\left\{r(x)^{1+\eta}, 1\right\}$ on $M$;
\item[(ii)] $|\nabla h|\leq Cr^{\eta}$ on $M\setminus\bar{B}_{2}(o)$;
\item[(iii)] $\left|\mathrm{Hess}\,h\right|(x)\leq Cr(x)^{2\eta}$ on $M\setminus\bar{B}_{2}(o)$
\end{itemize}
\end{theorem}

Finally note that the first part of Theorem \ref{HCOSubQuadr_coro} is equivalent to Theorem \ref{HCOSubQuadr} up to introduce the new function $H\in C^\infty(M)$ defined by $H=h^{\frac{1}{1+\eta}}$.

%\begin{remark}
%\rm{This exhaustion function meets the assumptions of Lemma \ref{HessCutOff} below, and hence gives the existence of a sequence of Hessian cut-off functions under these assumptions.}
%\end{remark}

\subsection{Sub-quadratic sectional curvature growth (no assumptions on injectivity radius)}

Let $o\in M$, $r(x)\doteq \mathrm{dist}_{g}(x,o)$ and let us assume that, for some $0<\eta\leq 1$ and $D>0$,
\begin{equation}\label{HpSect}
\ |\mathrm{Sect}_{g}|(x)\leq D^2(1+r(x)^2)^{\eta}.
\end{equation}
As in Step 0 of Subsection \ref{SubQuadrRic}, we start with the exhaustion function $h$ given in \cite{BianchiSetti}.\\
Given $R_{0}\in\mathbb{R}^{+}$, to be chosen later, and $x\in M$ such that $r(x)>1+R_0$,  we have that on $B_{R_0}(x)$
\[
\ |\mathrm{Sect}_{g}|\leq D^2(1+(R_0+r(x))^2)^{\eta}\doteq K_{x,R_0}.
\]
By a localized version of the Cartan-Hadamard theorem (see e.g. \cite[Lemma 2.7]{GuneysuPigola_17}) we have that for every $0<R<\min\left\{\pi/\sqrt{K_{x,R_0}},R_0\right\}$, there exists  a smooth complete Riemannian manifold $(\bar{M},\bar{g})$, $\bar{x}\in \bar{M}$, and a smooth surjective local isometry
\[
\ F\doteq F_{g,x,R}:B_{R}^{\bar{g}}(\bar{x})\to B_{R}^{g}(x), 
\]
such that
\begin{itemize}
\item $F(\bar{x})=x$;
\item $\mathrm{inj}_{\bar{g}}(\bar{x})\geq R$;
\item $|\mathrm{Sect}_{\bar{g}}|\leq K_{x,R_0}$ on $B_{R}^{\bar{g}}(\bar{x})$;
\item $F(B_{r}^{\bar{g}}(\bar{x}))=B_{r}^{g}(x)$, for all $0<r<R$.
\end{itemize}
In particular, for every $\bar{y}\in B^{\bar{g}}_{R/2}(\bar{x})$ ,we have that
\begin{equation}\label{Sect1}
|\mathrm{Sect}_{\bar{g}}|(\bar{y})\leq K_{x,R_0},\quad
\mathrm{inj}_{\bar{g}}(\bar{y})\geq d_{\bar{g}}\left(\bar{y},\partial B_{R}^{\bar{g}}(\bar{x})\right)\geq \frac{R}{2}.
\end{equation}
We define $\bar{h}_{x}:B_{R}^{\bar{g}}(\bar{x})\to\mathbb{R}$ by
\[
\ \bar{h}_{x}(\bar{y})=h(F(\bar{y}))-h(F(\bar{x})).
\]
Then $\bar{h}_{x}(\bar{x})=0$, and
\begin{itemize}
\item $|\bar{\nabla}\bar{h}_{x}|\leq C\bar{r}^{\eta}$,
\item $|\bar{\Delta}\bar{h}_{x}|\leq C\bar{r}^{2\eta}$,
\item $|\mathrm{Hess}^{\bar{g}}\bar{h}_{x}|_{\bar{g}}(\bar{y})=|\mathrm{Hess}^{g}h|_{g}(y)\,\mathrm{on}\,B_{R}^{\bar{g}}(\bar{x})$,
\end{itemize}
with $\bar{r}\doteq r\circ F$. Moreover, by \eqref{Poisson},
\begin{equation*}
\bar{\Delta}\bar{h}_{x}=|\bar{\nabla}\bar{h}_{x}|_{\bar{g}}^{2}-\theta(\tilde{r}\circ F)^{2\eta}\doteq\bar{f}.
\end{equation*}
Letting  $\lambda_{R_0}^2\doteq (m-1)K_{x,R_0}$, we set
\[
\ \bar{g}_{\lambda}=\lambda_{R_0}^2\bar{g}.
\]
Then, by \eqref{Sect1}
\begin{equation*}
|\mathrm{Ric}_{\bar{g}_{\lambda}}|(\bar{y})\leq 1,\quad
\mathrm{inj}_{\bar{g}_{\lambda}}(\bar{y})\geq \lambda_{R_0}\frac{R}{2}.
\end{equation*}
Assuming that $R_0\geq \left(\frac{\pi}{D}\right)^{\frac{1}{1+\eta}}$, we can take $R=\frac{\pi}{2\sqrt{K_{x,R_0}}}$, getting that
\[
\ \mathrm{inj}_{\bar{g}_{\lambda}}(\bar{y})\geq \frac{\sqrt{m-1}\pi}{4}\doteq i_{0}.
\]
Given $\alpha\in(0,1)$, $Q>1$ and $\delta>0$, Proposition \ref{HarmRadEst} hence yields that there exists a constant $C_{HR}(m,Q,\alpha,\delta, i_0)$ such that for every $\bar{y}\in B^{\bar{g}}_{\frac{R}{2}-\frac{\delta}{\lambda_{R_0}}}(\bar{x})$ we can find on $B^{\bar{g}}_{\frac{C_{HR}}{\lambda_{R_0}}}(\bar{y})$ harmonic coordinates with respect to which
\begin{itemize}
\item[(i)] $Q^{-1}\lambda_{R_0}^{-2}\delta_{ij}\leq \bar{g}_{ij}\leq Q\lambda_{R_0}^{-2}\delta_{ij}$;
\item[(ii)] $\sum_{s}\lambda_{R_0}^{2}C_{HR}\sup\left|\partial_{s}\bar{g}_{ij}(\bar{y})\right|+\sum_{s}C_{HR}^{1+\alpha}\lambda_{R_0}^{2-\alpha}\sup_{\bar{y}\neq \bar{z}}\frac{\left|\partial_{s}\bar{g}_{ij}(\bar{z})-\partial_{s}\bar{g}_{ij}(\bar{y})\right|}{d_{\bar{g}}(\bar{y},\bar{z})^{\alpha}}\leq Q-1$;
\end{itemize}
and thus also
\begin{itemize}
\item[(i')] $Q^{-1}\lambda_{R_0}^{2}\delta^{ij}\leq \bar{g}^{ij}\leq Q\lambda_{R_0}^2\delta^{ij}$;
\item[(ii')] $\sum_{s}\lambda_{R_0}^{-2}C_{HR}\sup\left|\partial_{s}\bar{g}^{ij}(y)\right|+\sum_{s}C_{HR}^{1+\alpha}\sup_{\bar{y}\neq \bar{z}}\lambda_{R_0}^{-2-\alpha}\frac{\left|\partial_{s}\bar{g}^{ij}(z)-\partial_{s}\bar{g}^{ij}(y)\right|}{d_{\bar{g}}(\bar{y},\bar{z})^{\alpha}}\leq C(Q)$,
\end{itemize}
for some constant C(Q). Traveling through again Step 2, 3, 4 and 5 of Subsection \ref{SubQuadrRic} we eventually get that
\[
\ |\mathrm{Hess}h|_{g}(x)=|\mathrm{Hess}\bar{h}_{x}|_{\bar{g}}(x)\leq C\bar{r}^{2\eta}(\bar{x})=Cr^{2\eta}(x).
\]
We have thus  proved the following
\begin{theorem}\label{HCOSubQuadrSect}
Let $(M, g)$ be a complete Riemannian manifold and $o\in M$ a fixed reference point, $r(x)\doteq \mathrm{dist}_{g}(x,o)$. Suppose that for some $0<\eta\leq1$ and $D>0$, 
\[
\ |\mathrm{Sect}_{g}|(x)\leq D^2(1+r(x)^2)^{\eta}.
\]
Then there exists an exhaustion function $h\in C^{\infty}(M)$ such that, for some positive constants $C>1$ independent of $x$ and $o$ and for some radius $R_0$, we have that
\begin{itemize}
\item[(i)]$C^{-1} r(x)^{1+\eta}\leq h(x)\leq C\max\left\{r(x)^{1+\eta}, 1\right\}$ on $M$;
\item[(ii)] $|\nabla h|\leq Cr^{\eta}$ on $M\setminus\bar{B}_{1+R_0}(o)$;
\item[(iii)] $\left|\mathrm{Hess}\,h\right|(x)\leq Cr(x)^{2\eta}$ on $M\setminus\bar{B}_{1+R_0}(o)$
\end{itemize}
\end{theorem}

Once again, defining $H\in C^\infty(M)$ by $H=h^{\frac{1}{1+\eta}}$, we get that the second part of Theorem \ref{HCOSubQuadr_coro} is equivalent to Theorem \ref{HCOSubQuadrSect}.

\section{Hessian cut-off functions}\label{HessCO}

In this section we construct (weak) Hessian cut-off functions starting from the distance-like functions obtained in the previous sections. This will permit to conclude the proof of Theorem \ref{th_main1}.

\begin{lemma}\label{HessCutOff}
Let $(M,g)$ be a complete Riemannian manifold and  $o\in M$ a fixed reference point. If there exists an exhaustion function  $h\in C^\infty(M)$ such that for some positive constants $D_{j}$, $\bar{\rho}>0$, $\beta>\varepsilon_{i}\geq 0$, we have that
 \begin{itemize}
 \item[(i)] $D_{1}r(x)^{\beta}\leq h(x) \leq D_{2} \max\left\{1,r(x)^{\beta}\right\}$ for every $x\in M$;
 \item[(ii)] $|\nabla h|(x)\leq D_{3}r(x)^{\beta-\varepsilon_{1}}$, for every $x\in M\setminus \bar{B}_{\bar{\rho}}(o)$;
\item[(iii)] $|\mathrm{Hess}(h)|(x)\leq D_{4}r(x)^{\beta-\varepsilon_{2}}$, for every $x\in M\setminus \bar{B}_{\bar{\rho}}(o)$.
 \end{itemize} 
Then given a $\gamma>(D_{2}/D_{1})^{1/\beta}$, there exists a family of cut-off functions $\left\{\chi_{R}\right\}$ such that
\begin{enumerate}
\item $\chi_R=1$ on $B_{R}(o)$ and $\chi_R=0$ on $M\setminus B_{\gamma R}(o)$;
\item $|\nabla \chi_R|\leq \frac{C_{1}}{R^{\varepsilon_{1}}}$;
\item $|\mathrm{Hess}(\chi_R)|\leq \frac{C_{2}}{R^{\min\left\{2\varepsilon_{1},\varepsilon_{2}\right\}}} $.
\end{enumerate}
In particular $\{\chi_n\}$ is a family of weak Hessian cut-off functions for every $\epsilon_1,\epsilon_2$ and of genuine Hessian cut-off functions whenever $\epsilon_1\epsilon_2>0$.
\end{lemma}

\begin{proof}
Let $\Gamma=\frac{D_{2}}{D_{1}}\geq 1$, and $\gamma>\Gamma^{\frac{1}{\beta}}$ a real number. Let $\phi\in C^\infty (\mathbb{R},[0,1])$ be such that 
\begin{equation*}
\begin{aligned}
\phi|_{(-\infty,\Gamma]}=1,\quad&\phi|_{[\gamma^{\beta},\infty)}=0,&|\phi'|+|\phi''|\leq a,
\end{aligned}
\end{equation*}
for some $a>0$. For any $R> 0$, let $\phi_R\in C^\infty ([0,+\infty))$ be defined by 
\[
\ \phi_R(t)\doteq\phi\left(\frac{t}{D_{1}R^{\beta}}\right).
\]
Then
\begin{equation*}
\begin{aligned}
|\phi^{\prime}_{R}|\leq \frac{a}{D_{1}R^{\beta}},\quad&|\phi^{\prime\prime}_R|\leq \frac{a}{D_{1}^2R^{2\beta}}.
\end{aligned}
\end{equation*}
For each radius $R\gg1$, define $\chi_R:=\phi_R\circ h$. Then it is immediate to verify that $\left\{\chi_{R}\right\}$ meets the required properties.
\end{proof}

We are finally in the position to give the 
\begin{proof}[Proof (of Theorem \ref{th_main1}).]
Under the assumptions (a) or (b) of the theorem, applying respectively Theorem \ref{HCOSubQuadr} or Theorem \ref{HCOSubQuadrSect}, we get the existence of a distance-like function $h$ with suitably controlled growth of the derivatives up to the 2nd order. Hence Lemma \ref{HessCutOff} applies and guarantees the existence of a sequence of (weak)-Hessian cut-off functions. Then Theorem \ref{th_main1} is a direct consequence of Proposition \ref{PropGP3.6}, (b).
\end{proof}

\section{The special case $p=2$}\label{p2}

\begin{proof}[Proof of Theorem \ref{Dens-p=2}]
Let $\lambda^{2}(r(x))\doteq-D^2 (1+r(x)^2)$. By \cite[Corollary 2.3]{BianchiSetti} we know that there exist a large constant $\gamma>1$ and a sequence  of weak Laplacian cut-off functions $\left\{\chi_{n}\right\}\subset C_{c}^{\infty}(M)$ such that
\begin{enumerate}
\item $\chi_{n}\equiv1$ on $B_{n}(o)$;
\item $\mathrm{supp}(\chi_{n})\subset B_{\gamma n}(o)$;
\item $|\nabla \chi_{n}|\leq \frac{C_{1}}{n}$;
\item $|\Delta\chi_{n}|\leq C_{2}$.
\end{enumerate}
with $C_{1}$ and $C_{2}$ independent of $n$. 

As noticed in Remark \ref{rmk_dens}, it is sufficient to consider $f\in C^{\infty}(M)\cap W^{2,2}(M)$. We want to prove that $\chi_{n}f$ converges to $f$ in $W^{2,2}(M)$. By properties (1) and (2) and the dominated convergence theorem it follows that
\[
\ \int_{M}|f-\chi_{n}f|^2d\mathrm{vol}_{g}\to 0,
\]
as $n\to\infty$. Furthermore, by properties (1), (2), (3), and the dominated convergence theorem, we have that
\begin{align*}
\int_{M}|\nabla f-\nabla (\chi_{n}f)|^2d\mathrm{vol}_{g}=&\int_{M}|\nabla f- (\chi_{n}\nabla f+f\nabla\chi_{n})|^2d\mathrm{vol}_{g}\\
=&C\int_{M}f^2|\nabla\chi_{n}|^2d\mathrm{vol}_{g}+C\int_{M}(1-\chi_{n})^2|\nabla f|^2d\mathrm{vol}_{g}\to 0,
\end{align*}
as $n\to\infty$. We now note that
\begin{equation*}
\int_{M}|\mathrm{Hess}(\chi_{n}f)-\mathrm{Hess}f|^2d\mathrm{vol}_{g}=\int_{M}(1-\chi_{n})^2|\mathrm{Hess}f|^2d\mathrm{vol}_{g}+2\int_{M}|\nabla\chi_{n}|^2|\nabla f|^2d\mathrm{vol}_{g}+\int_{M}|f|^2|\mathrm{Hess}\chi_{n}|^2d\mathrm{vol}_{g}.
\end{equation*}
Reasoning as above we get that the first two terms on the RHS converge to $0$. About the last term, using Bochner formula and our curvature assumption, we have that
\begin{align*}
\mathrm{div}\left(f^{2}\frac{\nabla|\nabla\chi_{n}|^2}{2}\right)=&f^{2}\frac{\Delta|\nabla\chi_{n}|^2}{2}+f\left\langle\nabla f,\nabla|\nabla\chi_{n}|^2\right\rangle\\
=&f^{2}\left[|\mathrm{Hess}\chi_{n}|^2+\mathrm{Ric}_{g}(\nabla\chi_{n},\nabla\chi_{n})+\left\langle\nabla\chi_{n},\nabla\Delta\chi_{n}\right\rangle\right]+2f|\nabla\chi_{n}|\left\langle\nabla f, \nabla |\nabla\chi_{n}|\right\rangle\\
\geq&f^{2}|\mathrm{Hess}\chi_{n}|^2+f^{2}\mathrm{Ric}_{g}(\nabla\chi_{n},\nabla\chi_{n})+f^{2}\left\langle\nabla\chi_{n},\nabla\Delta\chi_{n}\right\rangle-\frac{f^2}{2}|\nabla|\nabla\chi_{n}||^2-2|\nabla\chi_{n}|^2|\nabla f|^2\\
\geq&\frac{f^2}{2}|\mathrm{Hess}\chi_{n}|^2-2|\nabla\chi_{n}|^2|\nabla f|^2-\lambda^{2}f^{2}|\nabla\chi_{n}|^2+\mathrm{div}\left(f^{2}\Delta\chi_{n}\nabla\chi_{n}\right)\\
&-2f\Delta\chi_{n}\left\langle\nabla f,\nabla \chi_{n}\right\rangle-f^{2}(\Delta\chi_{n})^2\\
\geq&\frac{f^2}{2}|\mathrm{Hess}\chi_{n}|^2-3|\nabla\chi_{n}|^2|\nabla f|^2-\lambda^{2}f^{2}|\nabla\chi_{n}|^2+\mathrm{div}\left(f^{2}\Delta\chi_{n}\nabla\chi_{n}\right)-2f^{2}(\Delta\chi_{n})^2
\end{align*}
Integrating, we get that
\begin{equation}\label{p2Conv}
\frac{1}{2}\int_{M}f^{2}|\mathrm{Hess}\chi_{n}|^2d\mathrm{vol}_{g}\leq\int_{M}\lambda^{2}f^{2}|\nabla\chi_{n}|^2d\mathrm{vol}_{g}+2\int_{M}f^{2}(\Delta \chi_{n})^2d\mathrm{vol}_{g}+3\int_{M}|\nabla f|^2|\nabla\chi_{n}|^2d\mathrm{vol}_{g}.
\end{equation}
By property (3), and the dominated convergence theorem the last term on the RHS of \eqref{p2Conv}  converges to $0$ as $n\to\infty$. Moreover, by properties (1) and (2), we have that $\nabla \chi_{n}$ and $\Delta\chi_{n}$ are supported in $B_{\gamma n}(o)\setminus B_{n}(o)$. Hence, using property (3), the definition of $\lambda$, and the fact that $f\in L^{2}(M)$, we have that
\begin{align*}
\int_{M}\lambda^{2}|\nabla\chi_{n}|^2f^{2}d\mathrm{vol}_{g}=&\int_{B_{\gamma n}(o)\setminus B_{n}(o)}\lambda^{2}|\nabla\chi_{n}|^2f^2d\mathrm{vol}_{g}\\
\leq&\int_{M\setminus B_{n}(o)}C\frac{1+\gamma^{2}n^2}{n^2}f^{2}d\mathrm{vol}_{g}\leq \tilde{C}\int_{M\setminus B_{n}(o)}f^{2}d\mathrm{vol}_{g}\to0,
\end{align*}
as $n\to\infty$. Similarly, by property (4),
\[
\ \int_{M}f^{2}(\Delta \chi_{n})^2d\mathrm{vol}_{g}\leq C_{2}^{2}\int_{M\setminus B_{n}(o)}f^{2}d\mathrm{vol}_{g}\to 0,
\]
as $n\to\infty$. This finally gives our claim.
\end{proof}

\section{A disturbed Sobolev inequality}\label{sec_sob}

In this section we first prove the following general disturbed Sobolev inequality. Then we will deduce Theorem \ref{th_sob} from the proof of Theorem \ref{th_sob_dist}.
\begin{theorem}\label{th_sob_dist}
Let $(M^m,g)$ be a smooth, complete non-compact Riemannian manifold without boundary. Let $o\in M$, $r(x)\doteq \mathrm{dist}_{g}(x,o)$ and suppose  that one of the following set of assumptions holds
\begin{itemize}
\item[(a)] for some $0<\eta\leq1$, some $D>0$ and some $i_0>0$,
\[
\ |\mathrm{Ric}_{g}|(x)\leq D^2(1+r(x)^2)^{\eta},\quad\mathrm{inj}_{g}(x)\geq \frac{i_0}{D(1+r(x))^{\eta}}>0\quad\mathrm{on}\,\,M.
\]
\item[(b)] for some $0<\eta\leq1$ and some $D>0$,
\[
\ |\mathrm{Sect}_{g}|(x)\leq D^2(1+r(x)^2)^\eta.
\]
\end{itemize}
Let $p\in[1,m)$ and $q=mp/(m-p)$ and let $(\alpha,\beta)\in\mathbb R\times\mathbb R$ satisfying $\beta/p-\alpha/q\geq 1/m$. Then there exist constants $A_1>0$, $A_2>0$, depending on $m$, $p$, $\alpha$, $\beta$ and the constant $C$ from Theorem \ref{HCOSubQuadr_coro}, such that for all $\varphi\in C^\infty_c(M)$ it holds
\begin{align*}
\left(\int_M V_{\alpha} |\varphi|^{q}d\mathrm{vol}_{ g}\right)^{\frac{1}{q}}
&\leq A'\left(\int_M V_\beta|\nabla \varphi|^pd\mathrm{vol}_{g}\right)^{\frac{1}{p}} + B' \left(\int_M  V_\beta|\varphi|^p H^{p\eta}  d\mathrm{vol}_{g}\right)^{\frac{1}{p}}.
\end{align*}
Here
\[
V_\alpha(x)\doteq (r(x)+1)^{-\alpha m\eta}\left(\mathrm{vol}_g(B^g_{R_1}(x))\right)^{-\alpha},\quad \text{with } R_1=R_1(x)\doteq\begin{cases}C^{-\eta}(r(x)-1)^{-\eta}&if\ r(x)>1+C^{-1}\\1&otherwise,\end{cases},
\]
and
\[
V_\beta(x)\doteq(\max\{1;r(x)-1\})^{-\beta m\eta}\left(\mathrm{vol}_g(B^g_{R_2}(x))\right)^{-\beta},\quad \text{with }  R_2=R_2(x)\doteq\frac{C^{-\eta}}{2}\min\{1;r(x)^{-\eta}\}.
\]
\end{theorem}
\begin{remark}{\rm
As we will see in \eqref{bound_vol}, a lower bound on the injectivity radius implies a lower bound on the volumes. Accordingly, Theorem \ref{th_sob_dist} applies in more general situations with respect to Theorem \ref{th_sob}, e.g. under the assumption (b), or in case the volumes of geodesics balls of $(M,g)$ increase at infinity. On the other hand, in Theorem \ref{th_sob_dist} the value of $q=pm/(m-p)$ is  fixed.
}\end{remark}

\begin{proof}[Proof (of Theorem \ref{th_sob_dist})]
Let $\phi\in C^{\infty}(M)$ be a positive function to be choosen later, and define a new conformal metric $\tilde g = e^{2\phi}g$. Then, for any $X\in TM$, we have that 
\[
\mathrm{Ric}_{\tilde g}(X,X)=\mathrm{Ric}_{g}(X,X) - (m-2)\left[\mathrm{Hess}(\phi)(X,X)-g(X,\nabla\phi)^2\right]+g(X,X)\left(\Delta \phi - (m-2)|\nabla\phi|^2\right)
\]
Accordingly 
\[
|\mathrm{Ric}_{\tilde g}|\leq C(m) e^{-2\phi}\left\{|\mathrm{Ric}_{g}|+ [|\mathrm{Hess}(\phi)|+|\nabla \phi|^2]\right\}
\]
Now, let $H\in C^\infty(M)$ be the exhaustion function given by Theorem \ref{HCOSubQuadr_coro}. Without loss of generality we can suppose that $H>1$ on $M$. Indeed, if it is not the case, one can replace $H$ with a new function which approximates $\max\{2;H\}$ uniformly on $M$, and in $C^2$-norm outside a compact set. We recall the properties of $H$:  for some positive $C>1$,
\begin{itemize}
\item $\max\left\{1, C^{-1}r\right\}\leq H\leq C\max\left\{1, r\right\}$;
\item $|\nabla H|(x)\leq C$;
\item $|\mathrm{Hess}\,H|(x)\leq C\max\left\{r^{\eta}, 1\right\}$.
\end{itemize}
Choose $\phi=\eta\ln (H)$. In particular, since $e^\phi=H^\eta >1$, $(M,\tilde g)$ is complete. Moreover,
\[
|\nabla\phi|(x) = \eta\left|\frac{\nabla H}{H}\right|\leq \eta C\] and 
\[
|\mathrm{Hess} \phi|(x) = \eta\left|\frac{\mathrm{Hess}\, H}{H}-\frac{dH\otimes dH}{H^2}\right|\leq  2\eta C^2\max\left\{r(x)^{\eta},1\right\}. \]
Thus there exists a constant $\tilde C>0$ depending on $m, \eta$ and $C$ such that
 \begin{equation}\label{bound_ric}
|\mathrm{Ric}_{\tilde g}|(x)\leq \frac{C(m)}{H^{2\eta}}
\left\{\lambda(r(x))+ 2\eta C^2\max\left\{r(x)^{\eta}, 1\right\}+\eta^2 C^2\right\} \leq \tilde C,
\end{equation}
where we are still using the notation $\lambda(r(x))\doteq D^{2}(1+r^{2})^{\eta}$.

Set $v(x)=\left(\mathrm{vol}_{\tilde g}(B^\tg_1(x))\right)^{-1}$. Letting $p\in[1,m)$ and $q=mp/(m-p)$ and let $(\alpha,\beta)\in\mathbb R\times\mathbb R$ satisfying $\beta/p-\alpha/q\geq 1/m$, from \cite[Theorem 3.8]{HebeyCourant}, we have  the validity on $(M,\tg)$ of the disturbed Sobolev inequality 
\begin{equation}\label{dist-sob-th}
\left(\int_M |u|^{q}v^\alpha d\mathrm{vol}_{\tilde g}\right)^{\frac{1}{q}}\leq A \left(\int_M \tilde g(\tilde{\nabla} u,\tilde{\nabla} u)^{\frac{p}{2}}v^\beta d\mathrm{vol}_{\tilde g}\right)^{\frac{1}{p}} + B \left(\int_M  |u|^p v^\beta d\mathrm{vol}_{\tilde g}\right)^{\frac{1}{p}}
\end{equation}
for all $u\in C^{\infty}_c(M)$, and for some positive constants $A$ and $B$ independent of $u$.

Moving back to the metric $g$, this latter becomes
\[
\left(\int_M v^\alpha |u|^{q}e^{m\phi}d\mathrm{vol}_{ g}\right)^{\frac{1}{q}}\leq A \left(\int_M g(\nabla u, \nabla u)^{\frac{p}{2}}e^{(m-p)\phi}v^\beta d\mathrm{vol}_{g}\right)^{\frac{1}{p}} + B \left(\int_M  |u|^p e^{m\phi}v^\beta d\mathrm{vol}_{g}\right)^{\frac{1}{p}}
\]
for all $u\in C^{\infty}_c(M)$, i.e. 
\begin{equation}\label{Sob_conf}
\left(\int_M \left(|u|H^{\frac{m\eta}{q}}\right)^{q} v^\alpha d\mathrm{vol}_{ g}\right)^{\frac{1}{q}}\leq A \left(\int_M |\nabla u|^{p}H^{(m-p)\eta} v^\beta d\mathrm{vol}_{g}\right)^{\frac{1}{p}} + B \left(\int_M  |u|^p H^{m\eta} v^\beta d\mathrm{vol}_{g}\right)^{\frac{1}{p}}.
\end{equation}
In the following we will simplify the notations by writing  $|\nabla \cdot|^2$ for $g(\nabla \cdot,\nabla \cdot)$. Set $\varphi=uH^{\frac{m\eta}{q}}$. Then $u=\varphi H^{-\frac{m\eta}{q}}$ and 
\[|\nabla u|^2\leq H^{-\frac{2m\eta}{q}}\left(|\nabla \varphi|+\frac{m\eta}{q}\frac{|\varphi|}{H}|\nabla H|\right)^2,
\]
from which, using Jensen's inequality, we deduce that
\begin{align*}
\left|\nabla u\right|^p &\leq H^{-\frac{m\eta p}{q}}\left(|\nabla \varphi|+\frac{m\eta}{q}\frac{|\varphi|}{H}|\nabla H|\right)^p\\
&\leq 2^{p-1}H^{-\frac{m\eta p}{q}}|\nabla \varphi|^p+2^{p-1}\left(\frac{m\eta}{q}\right)^{p}|\varphi|^p H^{-\left(\frac{m\eta}{q}+1\right)p}|\nabla H|^p
\end{align*}
and
\begin{align*}
|\nabla u|^p H^{(m-p)\eta}
&\leq 2^{p-1}H^{\left(m-p-\frac{m p}{q}\right)\eta}|\nabla \varphi|^p+2^{p-1}|\varphi|^p\left(\frac{m\eta}{q}\right)^{p} H^{m\eta-p\eta-\frac{m\eta p}{q}-p}|\nabla H|^p\\
&\leq 2^{p-1}H^{\left(m-p-\frac{m p}{q}\right)\eta}|\nabla \varphi|^p+2^{p-1}|\varphi|^p\left(\frac{m\eta}{q}\right)^{p} H^{m\eta-\frac{m\eta p}{q}}|\nabla H|^p.
\end{align*}
Since $|\nabla H|$ is bounded, we finally get from \eqref{Sob_conf}
\begin{align*}
\left(\int_M |\varphi|^{q} v^\alpha d\mathrm{vol}_{ g}\right)^{\frac{1}{q}}&=
\left(\int_M \left(|u|H^{\frac{m\eta}{q}}\right)^{q} v^\alpha d\mathrm{vol}_{ g}\right)^{\frac{1}{q}} \\
&\leq A 2^{\frac{p-1}{p}} \left(\int_M \left[ H^{\left(m-p-\frac{m p}{q}\right)\eta}|\nabla \varphi|^p+|\varphi|^p\left(\frac{m\eta}{q}\right)^{p} H^{m\eta-\frac{m\eta p}{q}}|\nabla H|^p \right] v^\beta d\mathrm{vol}_{g}\right)^{\frac{1}{p}} \\
&+ B \left(\int_M  |\varphi|^p H^{m\eta-\frac{m\eta p}{q}} v^\beta d\mathrm{vol}_{g}\right)^{\frac{1}{p}}\\
&\leq A'\left(\int_M H^{\left(m-p-\frac{m p}{q}\right)\eta}|\nabla \varphi|^p v^\beta d\mathrm{vol}_{g}\right)^{\frac{1}{p}} + B' \left(\int_M  |\varphi|^p H^{m\eta-\frac{m\eta p}{q}} v^\beta d\mathrm{vol}_{g}\right)^{\frac{1}{p}},
\end{align*}
where $A'= A 2^{\frac{p-1}{p}}$ and $B'= A 2^{\frac{p-1}{p}}\frac{m\eta}{q}C+B$, with $C$ the constant in Theorem \ref{HCOSubQuadr_coro}. Note that $\frac{mp}{q}= m-p$, which in turn implies that $m\eta-\frac{m\eta p}{q}= p\eta$. Hence 
\begin{align}\label{sob_var}
\left(\int_M v^\alpha \left(|\varphi|\right)^{q}d\mathrm{vol}_{ g}\right)^{\frac{1}{q}}
&\leq A'\left(\int_M |\nabla \varphi|^pv^\beta d\mathrm{vol}_{g}\right)^{\frac{1}{p}} + B' \left(\int_M  |\varphi|^p H^{p\eta}v^\beta  d\mathrm{vol}_{g}\right)^{\frac{1}{p}}.
\end{align}
To conclude, we need the following lemmas.
\begin{lemma}\label{lem_low-est}
For all $x\in M$, $B^{\tg}_1(x)\subset B^g_{R_1}(x)$, with $R_1=R_1(x)=\begin{cases}C^{-\eta}(r(x)-1)^{-\eta}&if\ r(x)>1+C^{-1}\\1&otherwise.\end{cases}$
\end{lemma}
\begin{lemma}\label{lem_up-est}
For all $x\in M$, $B^{\tg}_1(x)\supset B^g_{R_2}(x)$, with $R_2=R_2(x)=\frac{C^{-\eta}}{2}\min\{1;r(x)^{-\eta}\}$.
\end{lemma}

According to Lemma \ref{lem_low-est}, and using the point-wise control on $H$, we have that
\begin{align*}
v^\alpha(x)\geq (\mathrm{vol}_{\tilde g}(B^g_{R_1}(x)))^{-\alpha} &= \left(\int_{B^g_{R_1}(x)} H^{m\eta}(y) d\mathrm{vol}_g(y)\right)^{-\alpha}\\
&\geq C^{-\alpha} (r(x)+1)^{-\alpha m\eta}\left(\mathrm{vol}_g(B^g_{R_1}(x))\right)^{-\alpha},
\end{align*}
while according to  Lemma \ref{lem_up-est}, we have that
\begin{align*}
v^\beta(x)\leq (\mathrm{vol}_{\tilde g}(B^g_{R_2}(x)))^{-\beta} &= \left(\int_{B^g_{R_2}(x)} H^{m\eta}(y) d\mathrm{vol}_g(y)\right)^{-\beta}\\
&\leq C^{\beta}(\max\{1;r(x)-1\})^{-\beta m\eta}\left(\mathrm{vol}_g(B^g_{R_2}(x))\right)^{-\beta}.
\end{align*}
Inserting the previous inequalities in \eqref{sob_var} concludes the proof of Theorem \ref{th_sob_dist}.

\end{proof}

It remains to prove Lemmas \ref{lem_low-est} and \ref{lem_up-est}.

\begin{proof}[Proof (of Lemma \ref{lem_low-est})]
Let $y\in B^{\tg}_1(x)$ and $\sigma:[0,a]\to M$ a minimizing geodesic of $(M,\tg)$ connecting $x$ and $y$. Let $a_0= \sup\{s\in[0,a]:\sigma([0,s])\subset B^g_{d_g(x,y)}(x)\}$. Then
\begin{align}\label{low_est}
1&>d_\tg(x,y)\\
&=\int_0^a \tg(\dot\sigma,\dot\sigma)^{1/2}(s)\,ds\nonumber\\
&\geq \int_0^{a_0} \tg(\dot\sigma,\dot\sigma)^{1/2}(s)\,ds\nonumber\\
&= \int_0^{a_0} e^{\phi}g(\dot\sigma,\dot\sigma)^{1/2}(s)\,ds\nonumber\\
&\geq \inf\{H^\eta(z):z\in B^g_{d_g(x,y)}(x)\}d_g(x,y)\nonumber.
\end{align}
Since $H(z)>\max\{1;C^{-1}r(z)\}$, we get 
\begin{equation}\label{low_est3}
d_g(x,y)<1,
\end{equation}
which prove the lemma when $r(x)<1+1/C$. Otherwise,
\begin{align*}
\inf\{H^\eta(z):z\in B^g_{d_g(x,y)}(x)\}
\geq C^{\eta} (r(x)-1)^{\eta}.
\end{align*}
Inserting in \eqref{low_est} gives
\begin{align}\label{low_est2}
d_g(x,y)< C^{-\eta} (r(x)-1)^{-\eta}.
\end{align}
This concludes the proof of the remaining case.
\end{proof}

\begin{proof}[Proof (of Lemma \ref{lem_up-est})]
Let $y\in B^{g}_{R_2}(x)$ and $\sigma:[0,a]\to M$ a minimizing geodesic of $(M,g)$ connecting $x$ and $y$. Then
\begin{align}\label{up_est}
R_2&>d_g(x,y)\\
&=\int_0^a g(\dot\sigma,\dot\sigma)^{1/2}(s)\,ds\nonumber\\
&= \int_0^{a} e^{-\phi}\tg(\dot\sigma,\dot\sigma)^{1/2}(s)\,ds\nonumber\\
&\geq \inf\{H^{-\eta}(z):z\in B^g_{R_2}(x)\}d_\tg(x,y)\nonumber\\
&\geq \inf\{C^{-\eta}\min\{1,r^{-\eta}(z)\}:z\in B^g_{R_2}(x)\}d_\tg(x,y)\nonumber.
\end{align}
Suppose first that $r(x)\geq 1+ R_2(x)$. Then $B^g_{R_2}(x)\subset M\setminus B_1^g(o)$. In particular, for all $z\in B^g_{R_2}(x)$ it holds $r^{-\eta}(z)\leq 1$, from which
\begin{align*}
\inf\{C^{-\eta}\min\{1,r^{-\eta}(z)\}:z\in B^g_{R_2}(x)\}&=\inf\{C^{-\eta}r^{-\eta}(z):z\in B^g_{R_2}(x)\}\\
&\geq C^{-\eta}\left[r(x)+R_2\right]^{-\eta}.
\end{align*}
The relation \eqref{up_est} implies in this case that
\begin{align*}
d_\tg(x,y)<& R_2C^\eta\left[r(x)+R_2\right]^{\eta}
=\frac{1}{2}\min\{1;r(x)^{-\eta}\}\left[r(x)+R_2\right]^{\eta}\\
\leq& \frac{1}{2}
r(x)^{-\eta}\left[r(x)+R_2\right]^{\eta}
\leq \frac{r(x)^{-\eta}}{2}\left[r(x)+\frac{1}{2}\right]^{\eta}<2^{\eta-1}\leq 1
\end{align*}
since $R_2\leq 1/2$.
On the other hand, if $r(x)< 1+ R_2(x)$, 
then for all $z\in B_{R_2}^g(x)$ one has
\begin{align*}
r(z)\leq r(x)+R_2\leq 1+2R_2 \leq 2. 
\end{align*}
In particular $\min\{1;r^{-\eta}(z)\}\geq 2^{-\eta}$ and \eqref{up_est} implies
\[ C^{-\eta}2^{-\eta}d_\tg(x,y)\leq R_2 \leq \frac{C^{-\eta}}{2},
\]
from which $d_\tg(x,y)\leq 2^{\eta-1}\leq 1$.
\end{proof}

\begin{proof}[Proof (of Theorem \ref{th_sob})]
Suppose first that $q=mp/(m-p)$. Since the assumptions of Theorem \ref{th_sob_dist} are satisfied, we have the validity of \eqref{sob_var}.

Reasoning as in the proof of Lemma \ref{lem_low-est}, one get the existence of $R>0$ such that for all $x\in M\setminus B^g_R(o)$, $\rho>0$, one has 
\begin{equation}\label{balls}
B^{g}_{\rho r^{-\eta}(x)}(x) \subset B^{\tilde g}_{2\rho}(x).
\end{equation}
Then
\begin{align}\label{comp-vol}
\mathrm{vol}_{\tilde g} (B^{\tilde g}_{2\rho}(x)) &
\geq \mathrm{vol}_{\tilde g}(B^{g}_{\rho r^{-\eta}(x)}(x))\\
&\geq \inf\{ e^{m\phi(z)}\ :\ z\in B^{g}_{\rho r^{-\eta}(x)}(x)\} \mathrm{vol}_{g} (B^{g}_{\rho r^{-\eta}(x)}(x))\nonumber\\
&\geq \left(r(x)-\rho r^{-\eta}(x)\right)^{m\eta}\mathrm{vol}_{g} (B^{g}_{\rho r^{-\eta}(x)}(x))\nonumber\\
&\geq \frac{r^{m\eta}(x)}{2} \mathrm{vol}_{g} (B^{g}_{\rho r^{-\eta}(x)}(x))\nonumber
\end{align}
for $r(x)$ large enough. Recall the following result by Croke.
\begin{lemma}[Proposition 14 in \cite{croke}]\label{lem_croke}
There exists a dimensional constant $C_m>0$ such that for any $x\in M$ and $i>0$, if
\[
\forall y\in B^g_{\frac{i}{2}}(x),\quad \mathrm{inj}_g(y)>i,
\]
then 
\[
\mathrm{vol}_g(B^g_{\frac{i}{2}}(x))\geq C_m i^m.
\]
\end{lemma}
Let $E=\min\{i_0/D;2\}$ and choose $i=i(x)=E(1+r(x))^{-\eta}$. There exists a positive radius $R_\eta$ large enough depending on $\eta$ such that for all $x\in M\setminus B_{R_{\eta}}^g(o)$ and for all $y\in B^g_{\frac{i}{2}}(x)$ we have
\[
\mathrm{inj}_g(y)\geq \frac{E}{(1+r(y))^{\eta}}\geq \frac{1}{2}\frac{E}{(1+r(x))^\eta}\geq i.
\]
Then Lemma \ref{lem_croke} applies and we get that for all $x\in M\setminus B_{R_{\eta}}^g(o)$
\[
\mathrm{vol}_g\left(B^g_{E(1+r(x))^{-\eta}/4}(x)\right)\geq C_m\frac {E^m}{2^m(1+r(x))^{m\eta}}.
\]
Choosing $\rho=E/4$ in \eqref{comp-vol}, we finally obtain that
\begin{align}\label{bound_vol}
\mathrm{vol}_{\tilde g} (B^{\tilde g}_{1}(x))
&\geq
\mathrm{vol}_{\tilde g} (B^{\tilde g}_{E/2}(x))\\
&\geq \frac{r^{m\eta}(x)}{2} \mathrm{vol}_{g} (B^{g}_{Er^{-\eta}(x)/4}(x))\nonumber\\
& \geq C_m \frac{r^{m\eta}(x)}{2} \frac {E^m}{2^m(1+r(x))^{m\eta}}\nonumber\\
&\geq \frac{C_mE^m}{2}\nonumber
\end{align}
if $r(x)$ is large enough. In particular, since $\mathrm{vol}_{\tilde g} (B^{\tilde g}_{1}(x))$ is continuous with respect to $x$ on $M$, it is uniformly lower bounded by a positive constant on the whole $M$. On the other hand, $\mathrm{vol}_{\tilde g} (B^{\tilde g}_{1}(x))$ is also uniformly upper bounded on $M$ by Bishop-Gromov theorem, since $\mathrm{Ric}_\tg$ is bounded.
In particular $v^\alpha$ is uniformly lower bounded by a positive constant, and $v^\beta$ is uniformly upper bounded, so that the conclusion follows from \eqref{sob_var}.

The case $q\in [p,mp/(m-p))$ can be treated similarly,  using the standard Sobolev inequality instead of \eqref{dist-sob-th}. Indeed, according to \cite{Varo,CS-C} (see also \cite[Theorem 3.2]{Hebey}), 
\[
\left(\int_M |u|^{q}d\mathrm{vol}_{\tilde g}\right)^{\frac{1}{q}}\leq A \left(\int_M \tilde g(\tilde{\nabla} u,\tilde{\nabla} u)^{\frac{p}{2}}d\mathrm{vol}_{\tilde g}\right)^{\frac{1}{p}} + B \left(\int_M  |u|^p d\mathrm{vol}_{\tilde g}\right)^{\frac{1}{p}}
\]
for all $u\in C^{\infty}_c(M)$. This latter is satisfied because of the bounds on the Ricci curvature and on the volumes of small balls given in \eqref{bound_ric} and \eqref{bound_vol}.
\end{proof}

\section{Some further applications}\label{FurtherAppl}
In this last section we highlight some additional applications of Theorem \ref{HCOSubQuadr_coro} and of the proof of Theorem \ref{th_sob}. 
\subsection{Full Omori-Yau maximum principle for the Hessian}
Recall that a Riemannian manifold $(M, g)$ is said to satisfy the full Omori-Yau maximum principle for the Hessian if for any function $u\in C^{2}(M)$ with $u^{*}=\sup_{M}u<+\infty$, there exists a sequence $\left\{x_{n}\right\}_{n}\subset M$ with the properties
\[
\ \mathrm{(i)}\,u(x_{k})>u^{*}-\frac{1}{k},\quad\mathrm{(ii)}\,|\nabla u (x_{k})|<\frac{1}{k},\quad\mathrm{(iii)}\,\mathrm{Hess}(u)(x_{k})<\frac{1}{k}g,
\]
for each $k\in\mathbb{N}$.

Letting $o$ be a fixed reference point in the complete Riemannian manifold $(M, g)$ and denoting by $r(x)$ the distance function from $o$, the full Omori-Yau maximum principle for the Hessian is known to hold e.g. if the radial sectional curvature of $M$ (i.e. the sectional curvature of $2$-planes containing $\nabla r$), satisfies
\[
\ K_{\mathrm{rad}}\geq -C^2(1+r^{2})\prod_{j=1}^{k}\left(\ln^{(j)}(r)\right)^2,
\]
where $\log^{(j)}$ stands for the $j$-th iterated logarithm; see \cite{PRS_Mem}.

The search of conditions weaker than a quadratic sectional curvature decay to $-\infty$ is a challenging problem; see e.g. \cite[Remark 5.6]{AMR_Book}. As a consequence of our results we can actually  guarantee the validity of the full Omori-Yau maximum principle for the Hessian under a Ricci quadratic bound and a linear injectivity radius decay. Namely, let $(M, g)$ be a complete Riemannian manifold, $o\in M$ a fixed reference point and $r(x)\doteq \mathrm{dist}_{g}(x,o)$. Suppose that for some $D>0$ and some $i_0>0$, we have that
\begin{equation}\label{HpFullOYHess}
\left|\mathrm{Ric}_{g}\right|(x)\leq D^2(1+r(x)^2),\quad\mathrm{inj}_{g}(x)\geq \frac{i_0}{D(1+r(x))}>0\quad\mathrm{on}\,\,M.
\end{equation}
Then the smooth exhaustion function $H$ constructed in Theorem \ref{HCOSubQuadr_coro} fulfill all the relevant requirements of \cite[Theorem 1.9]{PRS_Mem} with the choice $G(t)=1+t^2$. We hence obtain the validity of the following
\begin{corollary}
 Under the assumptions \eqref{HpFullOYHess}, the full Omori-Yau maximum principle for the Hessian holds on $(M,g)$.
  \end{corollary}
  
In this regard, one could also note that distance-like functions with controlled Hessian on an ambient space are actually inherited by properly immersed submanifolds with controlled growth of the second fundamental form. Indeed, using the same notations as above, suppose that $(M, g)$ is a complete manifold satisfying 
\begin{equation}\label{AssAmbPISubmfd}
\ |\mathrm{Ric}_{g}|(x)\leq D^2(1+r(x)^2)^{\eta},\quad\mathrm{inj}_{g}(x)\geq \frac{i_0}{D(1+r(x))^{\eta}}>0\quad\mathrm{on}\,\,M.
\end{equation}
and let $\varphi:(\Sigma,h)\to(M, g)$ be an isometric proper immersion. Then, letting $H$ be the distance-like function given by Theorem \ref{HCOSubQuadr_coro} on $(M, g)$, we have that the function $\gamma:\Sigma\to\mathbb{R}^{+}$ defined by
\[
\ \gamma\doteq \frac{H^2}{2}(\varphi)
\]
is still a proper function. Moreover, letting $\left\{e_{i}\right\}$ be a local orthonormal frame on $\Sigma$, we can compute 
\begin{align*}
|\nabla \gamma|^2=&\sum_{i}h(\nabla \gamma, e_{i})^2=(H(\varphi))^2\sum_{i}h(\nabla^{M} H, e_{i})^2\\
\leq&(H(\varphi))^2|\nabla^{M}H|^2\leq C (H(\varphi))^2=2C\gamma,
\end{align*}
Denoting by $(\cdot)^{\bot}$ the  projection on the normal bundle of $\Sigma$, and by $A(\cdot, \cdot)$ the second fundamental form of the immersion, we also have that, for any $X\in T\Sigma$,
\begin{align*}
\mathrm{Hess}\,\gamma(X, X)=&h(\nabla_{X}\nabla\gamma, X)\\
=&g(\nabla_{X}^{M}(H\nabla^{M} H- H(\nabla^{M}H)^{\bot}), X)\\
=&H\,\mathrm{Hess}\,^{M}(X,X)+(g(\nabla^{M}H, X))^2-Hg((\nabla^{M}H)^{\bot},A(X, X))\\
\leq&\left(|H||\mathrm{Hess}^{M}(H)|+|\nabla^{M} H|^2+|H||\nabla^M H||A|\right) |X|^2
\end{align*}
Hence, using the properties of $H$, an application of \cite[Theorem 1.9]{PRS_Mem} yields the validity of the following
\begin{corollary}
 Let $(M, g)$ be a complete Riemannian manifold, $o \in M$,  $r(x)\doteq \mathrm{dist}_{g}(x,o)$ and assume that assumptions \eqref{AssAmbPISubmfd} are satisfied. Let $\varphi:(\Sigma, h)\to (M, g)$ be an isometric proper immersion such that the second fundamental form $A$ satisfies
 \[
 \ |A|\leq Cr,
 \] 
outside a compact set, for some constant $C\in\mathbb{R}$. Then the full Omori-Yau maximum principle for the Hessian holds on $(\Sigma, h)$.
\end{corollary}

\subsection{Martingale completeness}
A probabilistic concept, introduced in \cite{Emery}, that in many instances seems to be related to maximum principles for the Hessian, is the martingale completeness. A manifold $(M, g)$ is called martingale complete if and only if each martingale on $(M, g)$ has infinite lifetime almost surely. Even though recent advances, using the language of subequations by Harvey-Lawson, can be found in \cite{MariPessoa}, the relation between this property and geometry has not been fully understood yet.  A classical criterion for its validity is given in \cite[Proposition 5.37]{Emery}: if  on $(M, g)$ there exists a positive function $f\in C^{2}$ such that
\[
\ \mathrm{(i)}f\,\,\mathrm{is}\,\,\mathrm{proper},\quad\mathrm{(ii)}|\nabla f|\leq C,\quad\mathrm{(iii)}\,\mathrm{Hess}(f)\leq Cg,
\]
for some constant $C>0$, then $(M, g)$ is martingale complete. It is hence a simple consequence of  \cite[Proposition 1.3]{RimoldiVeronelli} that a manifold with bounded Ricci curvature and positive injectivity radius is martingale complete. Up to our knowledge this was not yet observed in literature. One can then argue similarly to the previous subsection to transfer this property also to properly immersed submanifolds with bounded second fundamental form.

\subsection{An $L^{2}$-Calderon-Zygmund inequality with weight}
In this final subsection, reasoning as in Section \ref{sec_sob}, we give a proof of Theorem \ref{th_CZ_dist} stated in the Introduction.
\begin{proof}[Proof of Theorem \ref{th_CZ_dist}]
Let $H\in C^\infty(M)$ be the exhaustion function given by Theorem \ref{HCOSubQuadr_coro}. As in the proof of Theorem \ref{th_sob_dist}, without loss of generality we can suppose that, for some positive $C>1$,
\begin{itemize}
\item $\max\left\{1, C^{-1}r\right\}\leq H\leq C\max\left\{1, r\right\}$;
\item $|\nabla H|(x)\leq C$;
\item $|\mathrm{Hess}\,H|(x)\leq C\max\left\{r^{\eta}, 1\right\}$.
\end{itemize}

Choose $\phi=\eta\ln (H)$ and define a new conformal complete metric $\tilde g := e^{2\phi}g$ on $M$. We have in particular that  
$|\mathrm{Ric}_{\tilde g}|(x)$ is uniformly bounded on $M$; see \eqref{bound_ric}. According to \cite[Theorem B and Proposition 4.5]{GuneysuPigola}, one has the validity of the following Calderon-Zygmund inequality on $(M,\tg)$: for all $u\in C^\infty_c(M)$ and all $\epsilon>0$, it holds
\begin{align}\label{CZGP}
\||\widetilde{\mathrm{Hess}}\,u|_{\tg}\|^2_{\widetilde{L^2}}\leq \frac{\tilde C\epsilon^2}{2}\|u\|^2_{\widetilde{L^2}} + \left(1+\frac{\tilde C^2}{2\epsilon^2}\right)\|\tilde{\Delta} u\|^2_{\widetilde{L^2}}.
\end{align}
where $\tilde C$ is the constant in \eqref{bound_ric}. 

Here and in what follows recall that $\widetilde{\mathrm{Hess}}$, $\tilde \Delta$, $\tilde \nabla$ and $|\cdot|\tg$ are respectively the Hessian, Laplacian, covariant derivative and tensorial norm on $M$ relative to the metric $\tg$. Moreover we  are denoting $\|f\|^2_{\widetilde{L^2}}\doteq\int_M |f|^2 d\mathrm{vol}_{\tg}$.  All over this proof, $C_{i}$ will denote real  positive constants depending only on $m$, $\eta$, $D$ and the constant $C$ in Theorem \ref{HCOSubQuadr_coro}.

By standard computations, one has
\begin{equation}\label{conf_hess}
\widetilde{\mathrm{Hess}}\,u = e^{-2\phi}\left[ \mathrm{Hess}\,u - du\otimes d\phi - d\phi\otimes du + g(\nabla u, \nabla \phi) g, \right]
\end{equation}
from which 
\begin{align*}
e^{4\phi}|\widetilde{\mathrm{Hess}}\,u|_\tg^2 
&= |\mathrm{Hess}\,u|_g^2+2 |\nabla u|_g^2|\nabla \phi|_g^2 +(m-2)g(\nabla u,\nabla\phi)^2 - 2g(\nabla u,\nabla\phi)\Delta_gu - 4\mathrm{Hess}\,u(\nabla u,\nabla \phi)\\
&\geq |\mathrm{Hess}\,u|_g^2+2 |\nabla u|_g^2|\nabla \phi|_g^2 -(m-2)|\nabla u|_g^2|\nabla\phi|_g^2 - |\nabla u|_g^2|\nabla\phi|_g^2 - |\Delta_gu|^2\\ 
&- \frac{1}{2} |\mathrm{Hess}\,u|_g^2 - 8 |\nabla u|_g^2|\nabla\phi|_g^2. 
\end{align*}
Since $|\nabla \phi|_g\leq \eta C$ we obtain
\begin{align*}
|\widetilde{\mathrm{Hess}}\,u|_\tg^2 d\mathrm{vol}_\tg
\geq e^{(m-4)\phi} \left\{ \frac12 |\mathrm{Hess}\,u|_g^2 - (m+5)\eta^2 C^2|\nabla u|_g^2 -|\Delta_gu|^2 \right\}d\mathrm{vol}_g.
\end{align*}  
On the other hand, tracing \eqref{conf_hess} and proceeding as above, we get 
\begin{align*}
|\tilde\Delta\,u|^2 d\mathrm{vol}_\tg
\leq 2e^{(m-4)\phi} \left\{ |\Delta u|^2 + (m-2)^2 \eta^2 C^2|\nabla u|_g^2 \right\}d\mathrm{vol}_g.
\end{align*}  
Suppose that $\frac{m-4}{2}\eta\left(\frac{m-4}{2}\eta-1\right)\neq 0$, the other case being easier. Inserting the two last inequalities in \eqref{CZGP}, and recalling that $e^\phi=H^\eta$, we obtain
\begin{align}\label{CZu}
\|H^{(m-4)\eta/2}|\mathrm{Hess}\,u|_g\|^2_{L^2}
&\leq 2(m+5)\eta^2C^2\|H^{(m-4)\eta/2}|\nabla u|_g\|^2_{L^2} + 2\|H^{(m-4)\eta/2}\Delta u\|^2_{L^2} \\
&+\tilde C\epsilon^2\|H^{m\eta/2}u\|^2_{L^2} + 2\left(2+\frac{\tilde C^2}{\epsilon^2}\right)\|H^{(m-4)\eta/2}\Delta u\|^2_{L^2}\nonumber\\
&+2\left(2+\frac{\tilde C^2}{\epsilon^2}\right)(m-2)^2 \eta^2 C^2\|H^{(m-4)\eta/2}|\nabla u|_g\|^2_{L^2}\nonumber\\
&=\tilde C\epsilon^2\|H^{m\eta/2}u\|^2_{L^2} + C_1\|H^{(m-4)\eta/2}|\nabla u|_g\|^2_{L^2}+C_2\|H^{(m-4)\eta/2}\Delta u\|^2_{L^2}.\nonumber
\end{align}
%where $C_1=2(m+5)\eta^2C^2+\left(2+\frac{\tilde C^2}{\epsilon^2}\right)(m-2)^2 \eta^2 C^2$ and $C_2\doteq2\left(3+\frac{\tilde C^2}{\epsilon^2}\right)$.
Now, given $\varphi\in C^\infty_c(M)$, choose $u=H^{-(m-4)\eta/2}\varphi$. We can compute that
\[
\nabla\varphi = \frac{m-4}{2}\eta H^{(m-4)\eta/2-1}u\nabla H + H^{(m-4)\eta/2}\nabla u,
\]
\begin{align*}
\mathrm{Hess}\,\varphi=& \frac{m-4}{2}\eta\left(\frac{m-4}{2}\eta-1\right)H^{(m-4)\eta/2-2}u dH\otimes dH +\frac{m-4}{2}\eta H^{(m-4)\eta/2-1}u \mathrm{Hess}\,H \\
&+ \frac{m-4}{2}\eta H^{(m-4)\eta/2-1} (du\otimes dH + dH\otimes du) 
+ H^{(m-4)\eta/2}  \mathrm{Hess}\,u,
\end{align*}
and
\begin{align*}
\Delta\varphi
&= \frac{m-4}{2}\eta\left(\frac{m-4}{2}\eta-1\right)H^{(m-4)\eta/2-2}u|\nabla H|^2+ \frac{m-4}{2}\eta H^{(m-4)\eta/2-1}u\Delta H\\
& + (m-4)\eta H^{(m-4)\eta/2-1} g(\nabla u,\nabla H) + H^{(m-4)\eta/2}  \Delta u.
\end{align*}
In turn, since $|\nabla H|/H \leq C$, $|\mathrm{Hess}\,H|/H\leq C$ and $H^{(m-4)\eta/2}\leq H^{m\eta/2}$, this gives
\begin{align*}
\||\mathrm{Hess}\,\varphi|_g\|_{L^2}^2 \leq  4 \|H^{(m-4)\eta/2}|\mathrm{Hess}\,u|_g\|^2_{L^2} + C_3 \|H^{m\eta/2}u\|^2_{L^2} + C_4 \|H^{(m-4)\eta/2}|\nabla u|_g\|^2_{L^2},
\end{align*}
and
\begin{align*}
\|\Delta \varphi\|_{L^2}^2 \geq  4 \|H^{(m-4)\eta/2}\Delta u\|^2_{L^2} - C_5 \|H^{m\eta/2}u\|^2_{L^2} - C_6 \|H^{(m-4)\eta/2}|\nabla u|_g\|^2_{L^2},
\end{align*}
%where $C_3=***$, $C_4=***$, $C_5=***$ and $C_6=***$. 
Noticing also that 
\[
\|H^{(m-4)\eta/2}|\nabla u|_g\|^2_{L^2} \leq 2\||\nabla \varphi|_g\|^2_{L^2} + C_7 \|H^{m\eta/2}u\|^2_{L^2},
\]
from \eqref{CZu}, we thus get 
\begin{align*}
\||\mathrm{Hess}\,\varphi|_g\|_{L^2}^2 \leq  C_7 \|H^{2\eta}\varphi\|_{L^2}^2 +C_8 \||\nabla \varphi|\|_{L^2}^2 +C_9 \|\Delta \varphi\|_{L^2}^2.
\end{align*}
To conclude the proof, we note that by divergence theorem and Cauchy-Schwarz inequality one has
\begin{align*}
\||\nabla \varphi|\|_{L^2}^2 =\int_M |\nabla \varphi|^2 d\mathrm{vol}_g = -\int_M \varphi\Delta\varphi &\leq 2 \| \varphi\|_{L^2}^2 + 2\|\Delta \varphi\|_{L^2}^2
\\& \leq 2 \| H^{2\eta}\varphi\|_{L^2}^2 + 2\|\Delta \varphi\|_{L^2}^2.
\end{align*}

\end{proof}

\begin{acknowledgement*}
The first author is partially supported by INdAM-GNSAGA. The second and third authors are partially supported by INdAM-GNAMPA. We would like to thank the anonymous referees for their remarks and suggestions which have greatly improved the exposition and the quality of the paper. 
\end{acknowledgement*}

\end{document}